\documentclass{scrartcl}

\usepackage{amssymb,amsmath,amsthm,url,framed,amsfonts,enumitem}

\usepackage{url}
\usepackage{hyperref}
\usepackage{tikz}
\usetikzlibrary{arrows,shapes}
\usetikzlibrary {decorations.pathmorphing, decorations.pathreplacing, decorations.shapes,
} 
\allowdisplaybreaks

\usepackage{tcolorbox} 
  
 \let\OLDthebibliography\thebibliography
\renewcommand\thebibliography[1]{
  \OLDthebibliography{#1}
  \setlength{\parskip}{0pt}
  \setlength{\itemsep}{4pt plus 0ex}
}

\usepackage{fullpage}
\newcommand{\trw}{\mathsf{tw}} 
 
\newcommand{\cw}{\mathsf{cw}}

\newcommand{\iso}{\varphi} 
\newcommand{\st}{\mathsf{st}} 
\DeclareMathOperator{\ID}{Id}

\newcommand\inv{\operatorname{Inv}}

\renewcommand{\Pr}[1]{\mathbb{P}\!\left(\,#1\,\right)}

\newcommand{\Ex}[1]{\mathbb{E} \left[\,#1\,\right]}

\newcommand{\BO}[1]{\mathcal{O}\!\left(#1\right)}

\newcommand{\lo}[1]{o\!\left(#1\right)}

\newcommand{\eps}{\varepsilon}

\newcommand{\diam}{\mathsf{diam}}

\newcommand{\ber}[1]{\operatorname{Ber}\!\left(#1\right)}

\newcommand{\geo}[1]{\operatorname{Geo}\!\left(#1\right)}

\renewcommand{\leq}{\leqslant}
\renewcommand{\geq}{\geqslant}
\renewcommand{\le}{\leqslant}
\renewcommand{\ge}{\geqslant}
\renewcommand{\tilde}{\widetilde}
\renewcommand{\bar}{\overline}

\newcommand{\ce}{\mathcal{E}}

\newcommand{\cf}{\mathcal{F}}
\newcommand{\cp}{\mathcal{P}}

\newcommand{\NP}{\mathsf{NP}}

\allowdisplaybreaks

\newtheorem{theorem}{Theorem}[section]
\newtheorem{lemma}[theorem]{Lemma}
\newtheorem{corollary}[theorem]{Corollary}

\newtheorem*{claim}{Claim}
\newtheorem{remark}[theorem]{Remark}
\newenvironment{poc}{\begin{proof}[Proof of Claim]}{\end{proof}}

\usepackage{authblk}

\DeclareMathOperator{\layer}{\mathsf{layer}}

\begin{document}
\providecommand{\keywords}[1]{\textbf{\textit{Keywords---}} #1}
\providecommand{\codes}[1]{\textbf{\textit{AMS MSC 2010---}} #1}

\title{Tangled Paths: A Random Graph Model from Mallows Permutations\thanks{An extended abstract of this paper appeared at EUROCOMB 2023 \cite{euro}.}}
 \date{\vspace*{-6mm}}
\author[1]{Jessica Enright}
\author[1]{Kitty Meeks}
\author[1]{William Pettersson}
\author[1,2]{John Sylvester}
\affil[1]{School of Computing Science, University of Glasgow, UK\\
	\texttt{\{jessica.enright,kitty.meeks,william.pettersson\}@glasgow.ac.uk}}
\affil[2]{Department of Computer Science, University of Liverpool, UK\\ \texttt{john.sylvester@liverpool.ac.uk}} 
\maketitle

\begin{abstract}
	We introduce the random graph $\mathcal{P}(n,q)$ which results from taking the union of two paths of length $n\geq 1$, where the vertices of one of the paths have been relabelled according to a Mallows permutation with  parameter $0<q(n)\leq 1$. This random graph model, the tangled path, goes through an evolution: if $q$ is close to $0$ the graph bears resemblance to a path, and as $q$ tends to $1$ it becomes an expander. In an effort to understand the evolution of $\mathcal{P}(n,q)$ we determine the treewidth and cutwidth of $\mathcal{P}(n,q)$ up to log factors for all $q$. We also show that the property of having a separator of size one has a sharp threshold. In addition, we prove bounds on the diameter, and vertex isoperimetric number for specific values of $q$.  
\end{abstract}
\\
\codes{05C80, 05A05, 68Q87, 05C78.}
\section{Introduction}
Given two graphs $G,H$ on a common vertex set $[n]=\{1,\dots, n \}$, and a permutation $\sigma$ on $[n]$, it is natural to consider the following graph \[\layer(G, \sigma(H))=\left([n],\; E(G) \cup\left\{\sigma(x)\sigma(y) : xy \in E(H)\right\}\right),\]which is the union of two graphs where the second graph has been relabelled by a permutation $\sigma$. Constructions of graphs via unions are very natural, and have appeared in several contexts, see Section \ref{sec:related}. Let $P_n$ be the path on $[n]$ that connects $i$ to $i+1$, for $i \in [n-1]$, and let $S_n$ be the set of all permutations on $[n]$. Consider the following scenario: one must choose a permutation $\sigma\in S_n$ with the goal of making $\layer(P_n,\sigma(P_n))$ as different from a path as possible. There are many parameters one may use to measure the difference between a connected graph $G$ and a path; for example one may look at the diameter $\diam(G)$ or the vertex isoperimetric number $\iso(G)$, as the path is extremal for these parameters. The treewidth $\trw(G)$ which, broadly speaking, measures how far (globally) the graph is from being a tree \cite{Kloks}, is another natural candidate. Given two or more paths  one can build a grid-like graph (see \cite[Lemma 8]{TwoDichot} for more details) which would have treewidth and diameter $\Theta(\sqrt{n})$. If we choose a permutation uniformly at random, then as a consequence of a result of Kim \& Wormald \cite[Theorem 1]{KW}, with high probability the resulting graph is a bounded degree expander. Thus, in this case, the graph $\layer(P_n,\sigma(P_n))$ has treewidth $\Theta(n)$ and diameter $\Theta(\log n)$, so by these parameters it is essentially as far from a path as a sparse graph can be.  

  The example above shows that even restricting the input graphs to paths can produce rich classes of graphs. Having seen what happens for a uniformly random permutation, one may ask about the structure of $\layer(P_n,\sigma(P_n))$ when $\sigma$ is drawn from a distribution on $S_n$ that is not uniform. One of the most well known non-uniform distributions on $S_n$ is the Mallows distribution, introduced by Mallows \cite{Mallows} in the late 1950s in the context of statistical ranking theory.  Recently it has been the subject of renewed interest for other applications, and as an interesting and natural model to study in its own right, see Section \ref{sec:related}. The distribution has a parameter $q$ which, roughly speaking, controls the amount of disorder in the permutation.
 
 For real $q>0$ and integer $n\ge 1$, the $(n,q)$-\textit{Mallows measure} $\mu_{n,q}$ on $S_n$ is given by
\begin{equation}\label{eq:mu_n_q_def}
\mu_{n,q}(\sigma) = \frac{q^{\inv(\sigma)}}{Z_{n,q}}\qquad\text{for any }\sigma \in S_n, 
\end{equation}
where $\inv(\sigma) = |\{(i,j)\,:\,\text{$i<j$ and $\sigma(i)>\sigma(j)$}\}|$ is the number of inversions in the permutation $\sigma$ and $Z_{n,q}$ is given explicitly by the following
formula \cite[Equation (2)]{BhatSubSeq}:
\begin{equation*}
Z_{n,q} =\prod_{i=1}^{n}\left(1 + q + \cdots + q^{i-1} \right)=  \prod_{i=1}^{n}\frac{1-q^{i}}{1-q}.
\end{equation*}  
When $q\rightarrow 0$, the distribution $\mu_{n,q}$ converges weakly to the degenerate distribution on the identity permutation. We extend $\mu_{n,q}$ to $q=0$ by setting $\mu_{n,0}$ to be the probability measure assigning $1$ to the identity permutation. On the other hand if $q=1$ then $\mu_{n,1}$ is the uniform measure on $S_n$. One can see that $\sigma \sim \mu_{n,q}$ has distribution $\mu_{n,1/q}$ when reversed.  

We study the random graph given by $\layer(P_n, \sigma(P_n))$, where $\sigma \sim \mu_{n,q}$ and $0\leq q:=q(n)\leq 1$. From now on we call this random graph the \textit{tangled path} model and denote it by $\cp(n,q)$. Thus a random graph $\cp(n,q)$ has vertex set $[n]$ and (random) edge set $E(P_n)\cup \{\sigma(i)\sigma(i+1) : i\in [n-1] \}$, where $\sigma \sim \mu_{n,q}$. We restrict to $q\in [0,1]$ as reversing the permutation does not affect our construction (up to a relabelling, see \eqref{reversed}). We also identify any multi-edges created as one edge, however this detail is not important for any of our results. This paper will focus on $\cp(n,q)$; as we have seen already combining paths can give rise to interesting and varied graphs, and the Mallows permutation gives our model a parameter $q$ which, roughly speaking, increases the `tangled-ness' of the graph. Other reasons for using Mallows permutations are that they are well studied (see Section \ref{sec:related}), and they are mathematically tractable since they can be generated by a sequence of independent random variables (see Section \ref{sec:Mallowsproc}).

By the above, $\cp(n,0)$ is a path and $\cp(n,1)$ is an expander with high probability; the latter follows from \cite[Theorem 1]{KW} but we also give a self-contained proof in this paper. Our ultimate aim is to understand the structure of $\cp(n,q)$ for intermediate values of $q$, and this paper takes the first steps in this direction. Informally, if $q$ is not tending to $1$ too fast, then $\cp(n,q)$ is `path-like'; we show that if $q<1$ is fixed the diameter is linear (Theorem \ref{diam}), and there is a sharp threshold for having a single cut vertex at $q_c=1-\frac{\pi^2}{6\log n} $ (Theorem \ref{prop:sep}). For $q\rightarrow 1$ sufficiently fast, it makes more sense to measure the complexity of the internal structure of $\cp(n,q)$ by how much it differs from a tree. Here we show that, up to logarithmic factors, the treewidth \cite{Kloks} of $\cp(n,q)$ grows at rate $(1-q)^{-1}$ (Theorem \ref{thm:treepathcut}) until the graph becomes an expander at around $q=1-\frac{1}{n} $ (Theorem \ref{treebanddiam}), indicating that, in the sense of treewidth, the complexity of the structure grows smoothly with $q$. This behaviour contrasts with the binomial/Erd\H{o}s-R\'{e}nyi random graph \cite{FriezeKaronski} where treewidth increases rapidly from being bounded by a constant, to $\Theta(n)$ as the average degree rises from below one to above one \cite{do2022note,LeeLO12}.

 Further motivation for this line of study comes from practical algorithmic applications.  Many real-world systems -- including social, biological and transport networks -- involve qualitatively different types of edges, where each type of edge generates a ``layer'' with specific structural properties \cite{Multi1,Multi2}.  For example, when modelling the spread of disease in livestock, one layer of interest arises from physical adjacency of farms, and so is determined entirely by geography.  A second epidemiologically-relevant layer could describe the pairs of farms which share equipment: this is no longer fully determined by geography, but will nevertheless be influenced by the location of farms, as those that are geographically close are more likely to cooperate in this way.  It is known that algorithmically useful structure in individual layers of a graph is typically lost when the layers are combined adversarially \cite{TwoDichot}. The present work can be seen as an attempt to understand the structure of graphs generated from two simple layers which are both influenced to some extent by a shared underlying ``geography''. In this setting the treewidth $\trw$ is a natural parameter as many $\NP$-hard problems become tractable when parametrised by $\trw$ \cite[Ch.\ 7]{Cygan}.

\subsection{Our Results}\label{sec:intro}

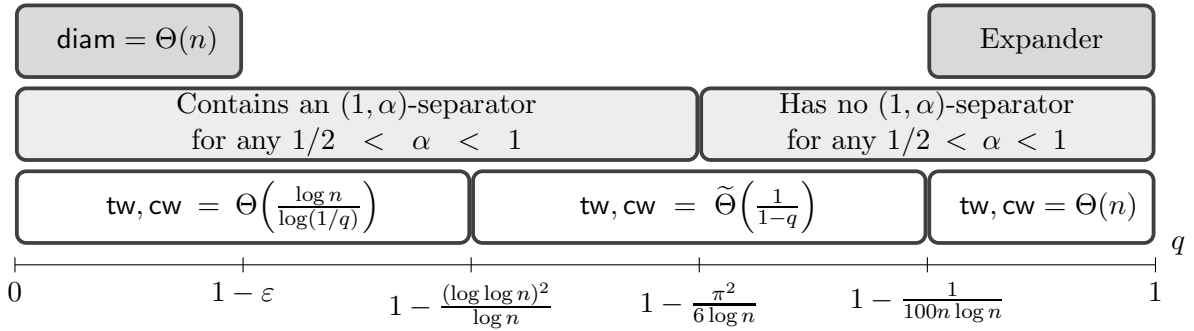
\begin{figure}
\begin{center}
\begin{tikzpicture} 

\draw[-] (0,0) to node[pos=1.02,above]{$q$} (15,0);

\draw (0,0.1) to node[pos=1,below]{$0$}   (0,-0.1);

\draw (15,0.1) to node[pos=1,below]{$1$}   (15,-0.1);

\draw (3,0.1) to node[pos=1,below]
{$1-\eps$} 
  (3,-0.1);

\draw (6,0.1) to node[pos=1,below]{
	$1-  \frac{(\log\log n)^2}{\log n } $ 
}  (6,-0.1);

\draw (9,0.1) to node[pos=1,below]{ $1-\frac{\pi^2}{6\log n } $ 
} node[pos=0,above]{} (9,-0.1);

\draw (12,0.1) to node[pos=1,below]{$1-\frac{1}{50n}$} node[pos=0,above]{} (12,-0.1);

 	\draw (3,1.4) node[anchor=north]{\begin{tcolorbox}[ width=6cm,height=1cm, halign=center,valign=center,colback=white ]
 	$\trw,\cw = \Theta\!\left(\sqrt{\frac{\log n}{\log (1/q)}}\right)$ 
 		\end{tcolorbox}};
 	
 	\draw (9,1.4) node[anchor=north]{\begin{tcolorbox}[ width=6cm,height=1cm,  halign=center,valign=center,colback=white]
 	  $\trw,\cw = \tilde{\Theta}\!\left(\frac{1}{1-q}\right)$ 
 		\end{tcolorbox}};

	\draw (13.5,1.4) node[anchor=north]{\begin{tcolorbox}[ width=3cm,height=1cm, halign=center,valign=center,colback=white]
	\mbox{$\!\!\trw,\cw = \Theta(n)$}
	\end{tcolorbox}};


\draw (4.5,2.5) node[anchor=north]{\begin{tcolorbox}[ width=9cm,height=1cm, halign=center,valign=center,colback=black!7]
	Contains an $(1,\alpha)$-separator for any $1/2<\alpha < 1$ 
	\end{tcolorbox}};

\draw (12,2.5) node[anchor=north]{\begin{tcolorbox}[ width=6cm,height=1cm, halign=center,valign=center,colback=black!7]
	Has no $(1,\alpha)$-separator for any $1/2<\alpha < 1$
	\end{tcolorbox}};


\draw (1.5,3.6) node[anchor=north]{\begin{tcolorbox}[ width=3cm,height=1cm, halign=center,valign=center,colback=black!15]
	\mbox{$ \diam = \Theta(n)$} 
	\end{tcolorbox}};

\draw (13.5,3.6) node[anchor=north]{\begin{tcolorbox}[ width=3cm,height=1cm, halign=center,valign=center,colback=black!15]
Expander
	\end{tcolorbox}};

\end{tikzpicture}
\vspace*{-6mm}
\end{center}	 
\caption{\label{fig:resultslinething}The diagram above gives a representation of our main results. All results above hold with high probability, and we say that $f(n) = \tilde{\Theta}(g(n)) $ if there exist constants $c,C>0 $ and $n_0$ such that $c\cdot g(n)/\log g(n)\leq f_n \leq C\cdot  g(n)\cdot \log g(n)$ for all $n\geq n_0$.   }
\end{figure}
In what follows, the integer $n\geq 1 $ denotes the number of vertices in the graph (or elements in a permutation) and $q:=q(n)$, the parameter of the Mallows permutation (or related tangled path), is a real-valued function of $n$ taking values in $[0,1]$. We say a sequence of events $\ce_n$ occurs \textit{with high probability} (w.h.p.) if $\Pr{\ce_n}\rightarrow 1$ as $n\rightarrow \infty$. Throughout $\log$ is base $\mathrm{e}$. See Figure \ref{fig:resultslinething} for a summary of our results. 

A graph $G$ is a \textit{vertex-expander} if there exists $c>0$ (independent of $n$) such that any set $S\subseteq V$ with $|S|\leq \lceil n/2\rceil $ is adjacent to at least $c|S|$ vertices in $V\backslash S$, see Section \ref{sec:widthdef}. As mentioned above, when $q=1$ the permutation is uniform, and so the fact that w.h.p.\ $\cp(n,1)$ is an expander follows from \cite[Theorem 1]{KW}. We give a self-contained proof of this fact,  which also shows that for $q$ sufficiently close to $1$, this still holds.

\begin{theorem}\label{treebanddiam}
 If $ q \geq 1 - \frac{1}{50n}$, then w.h.p. $\cp(n,q)$ is a  vertex-expander.  \end{theorem}

For an integer $s\geq 1 $ and $1/2\leq \alpha< 1$ we say that a graph $G$ with vertex set $V$ has an $(s,\alpha)$\textit{-separator} if there is a vertex subset $S$ with $|S|\leq s$ such that $V\backslash S $ can be partitioned into two disjoint sets of at most $\alpha|V|$ vertices with no crossing edges, see Section \ref{sec:notate}. Balanced separators (e.g.\  $\alpha=2/3$) are useful for designing divide and conquer algorithms, in particular for problems on planar graphs \cite{TarLip}, and are connected to other notions of sparsity \cite{Sparsity}.

Observe that, for any fixed $1/2<\alpha<1$, if $G$ is a vertex expander then there exists a $c>0$ such that $G$ has no $(c n,\alpha)$-separator. At the other extreme, the path has a $(1,\alpha)$\textit{-separator}. We show that for $\cp(n,q)$ this `path-like' property disappears around $q_c=1-\frac{\pi^2}{6\log n}$.

		\begin{theorem}\label{prop:sep} For any fixed $1/2< \alpha <1 $ we have  
			\[\lim_{n\rightarrow \infty}\Pr{\cp(n,q) \text{ has a }(1,\alpha )\text{-separator}} = \begin{cases}
			0 &\text{ if } \frac{\pi^2}{6(1-q)} - \log n + \frac{5\log\log n}{2} \rightarrow \infty  \\
			1&\text{ if } \frac{\pi^2}{6(1-q)} - \log n + \frac{9\log\log n}{2} \rightarrow -\infty 
			\end{cases}.  \]
		\end{theorem} 
	
	In other words, Theorem \ref{prop:sep} shows that if $q$ is sufficiently below the threshold $q_c=1-\frac{\pi^2}{6\log n}$ then w.h.p.\ there is a cut vertex which separates the graph into two $\Theta(n)$ vertex subpaths, and if $q$ is sufficiently above it then there does not. We say that $q_0$ is \textit{sharp threshold} for a graph property $\mathfrak{P}$ if for any $\eps>0$ w.h.p.\ $\cp(n,p) \notin \mathfrak{P}$ for any $p\leq q_0(1-\eps)$, and $\cp(n,r) \in \mathfrak{P}$ for any $r\geq q_0(1+\eps)$, see \cite{FriKal}. Theorem \ref{prop:sep} is quite precise as it determines the second order of the threshold up to a constant, showing that the property of having a $(1,\alpha )$-separator has a sharp threshold of width $\BO{\frac{\log\log n}{(\log n)^2}}$. Theorem \ref{prop:sep} is established by finding first and second moment thresholds for the property. Positive correlation between cuts suggests this result cannot be significantly improved using standard methods alone (see Remark \ref{rmk:sep}).

The diameter $\mathsf{diam}(G)$ of a graph $G$ is the length of the longest shortest path between any pair of vertices. Theorem \ref{treebanddiam} implies that $\diam(\cp(n,q)) =\BO{\log n}$ when $q$ is sufficiently close to $1$. On the other hand $ \diam(\cp(n,0))=n-1$ as $\mathcal{P}(n, 0)$ is a path; we show this holds (up to a constant) for any fixed $q<1$.

\begin{theorem} \label{diam} For any $0<\eps<1$, let  $0 < q \leq  1-\eps$. Then, there exists a constant $c:=c(\eps)>0$ such that for $n\geq 1/c$, we have $\Pr{ \mathsf{diam}\!\left(\mathcal{P}\!\left(n,q\right) \right)\geq cn}\leq n^{-1/10} .$
\end{theorem}
This result follows from bounds on the number of cut vertices used to prove Theorem \ref{prop:sep}.

The treewidth $\trw(G)$ of a graph $G$  is one less than the minimum size of the largest vertex subset (i.e.\ bag) in a tree decomposition of $G$, minimised over all such decompositions. The cutwidth $\cw(G)$ is the greatest number of edges crossing any real point under an injective function $f:V\rightarrow \mathbb{Z}$, minimised over all $f$. See Section \ref{sec:widthdef} for full definitions of these quantities. It is known that for any graph $G$ we have $\trw(G) \leq \cw(G)$, however there may be a multiplicative discrepancy of order up to $n$. We show there is at most only a constant factor discrepancy for $G=\cp(n,q)$ in certain ranges of $q$, and give bounds for all $q$ which are tight up to a $\log$ factor.

\begin{theorem}\label{thm:treepathcut}For any constant $\kappa>0$, let $0< q \leq  1 - \kappa\cdot \frac{(\log\log n)^2 }{\log n}$. Then, there exist constants $c_1,c_2>0$ such that   w.h.p.\ \[c_1\cdot \left(\sqrt{\frac{\log n}{\log (1/q)} } + 1\right)  \leq  \trw(\cp(n,q))\leq \cw(\cp(n,q)) \leq c_2\cdot\left( \sqrt{\frac{\log n}{\log (1/q)} } + 1 \right).\]
	Furthermore, if $ 1 - \frac{(\log\log n)^2 }{\log n} \leq q \leq 1$, then  w.h.p.\  
	\[ 10^{-5}\cdot  \min\left\{\frac{1}{1-q}, \; n \right\}\leq  \trw(\cp(n,q)) \leq \cw(\cp(n,q))\leq  5 \cdot \max \left\{  \frac{1}{1-q}\cdot \log\left(\frac{1}{1-q} \right) , \;  n \right\}.\]
\end{theorem}

Observe that if $q\rightarrow 1 $ then $\log(1/q) \approx 1-q$ and so when $q= 1-\Theta\left( (\log\log n)^2/\log n \right)$ we have $\sqrt{\log(n)/\log (1/q) } \approx  -\log(1-q) / (1-q) $. Thus, the two upper bounds on the cutwidth are equal up to constants for this range of $q$. Hence, for this range of $q$, the upper bound for the cutwidth given in the second equation is tight and the lower bound for treewidth is off by a multiplicative factor of order $\log\log n$. 
 
 The lower bounds on treewidth in Theorem \ref{thm:treepathcut} are proved by relating the treewidth to the occurrence of certain permutations as consecutive patterns in the underlying Mallows permutation. The upper bounds on cutwidth are proved by controlling the density of long edges.

\subsection{Further Related Work}\label{sec:related}
Many works have studied properties of a typical permutation sampled from the Mallows measure, in particular the longest increasing subsequence \cite{BhatSubSeq,Starr}, cycle structure \cite{PelCycle}, permutation pattern avoidance \cite{Patterns,PinskyPattern} and sets of consecutive elements \cite{pinsky}. Mallows permutations also arise as limit objects from stable matchings \cite{StabMatch} and have been studied in the contexts of statistical physics \cite{StarPhysic},  Markov chains \cite{DiaRam}, learning theory \cite{Braver} and finitely dependent processes \cite{FiniteDep}.

Random graphs have been heavily studied since their introduction in the late 1950s  \cite{FriezeKaronski}. A random graph arising from the Mallows distribution is introduced in \cite{BM}. In this model each edge corresponds to an inversion in the permutation, so it is different to our model. To our knowledge the model in \cite{BM} and the tangled path introduced here are the only random graph models based on Mallows permutation. However, some works have studied relations between random graph models and uniform permutations. In particular, in \cite{KW} it is proven that the union of two uniformly permuted cycles is contiguous to a random $4$-regular graph. Very recently \cite{chew2024hardness} used a union of two paths permuted by a uniform random permutation to get a lower bound on resolution refutations for SAT solvers. Also,  \cite{FriezeSpan} shows the union of two uniformly random trees on the same vertex set is an expander with high probability. 

There are also several papers which consider the graphs formed from (permuted) unions of graphs. In particular independent sets in the union of two Hamiltonian cycles \cite{Aharoni}, and the treewidth of a union of two graphs glued using a permutation \cite{viktor} and \cite[Chapter 5]{Qthesis}. The clique number of graph unions \cite{AharoniBCZ15,OthmanB18}, and unions of cliques have been studied \cite{MR2218735}. Unions of dense graphs with random graphs, namely `randomly perturbed graphs', have been studied intensely, see \cite{BohmanFM03} and citing papers. There is also a connection between graph unions and threshold graphs \cite{HoangT00}.  From the other direction, decompositions of graphs have been well studied \cite{Glock,Mont,Tutte}.

\subsection{Outline of the Paper}
In Section \ref{sec:notate} we cover some basic notation, definitions and concentration inequalities. In Section \ref{sec:mallowtangle} we state some known facts about the Mallows distribution, in particular defining the $q$-Mallows process, before introducing our notions of `flushing' and `local' events that are useful later in the paper. Section \ref{sec:q=1} establishes properties of the tangled path in the case where $q$ is close to one. The first result in this section shows that when $q=1$ (i.e.~a uniformly random permutation) the tangled path is an expander with high probability. Then, a bound on the probability of events under the $q$-Mallows measure by that of the $1$-Mallows measure is shown; this is useful later when bounding the treewidth. The beginning of Section \ref{sec:sep} focuses on  bounding the probability of flushing events. In the remainder of Section \ref{sec:sep}, these bounds are used to prove a sharp threshold for $(1,\alpha)$-separators and a linear bound on the diameter. Arguably the most interesting techniques and proofs are in Section \ref{sec:treewidth}. To prove a lower bound on the treewidth we use consecutive patterns in the Mallows permutations to find smaller tangled paths with a higher $q$ parameter as minors in the tangled path. To prove corresponding upper bounds we control the cutwidth by bounding the number of `long' edges created during the $q$-Mallows process. Finally, we conclude with some open problems in Section \ref{sec:probs}.

\section{Notation and Preliminaries} \label{sec:notate}

For a random variable $X$ and probability measure $\mu$ we use $X\sim \mu$ to say that $X$ has distribution $\mu$. Let $X,Y$ be random variables, then $X$ \textit{stochastically dominates} $Y$ if $ \Pr{ X \geq x } \geq   \Pr{ Y \geq x }$ for all real $x$, and we denote this by $X\succeq Y$. 
We let $\Omega$ denote the sample space and $\ce^c=\Omega\backslash \mathcal{E}$ to denote the complement of an event $\ce$. We also let $\mathbf{1}_{\ce}$:$\Omega\rightarrow \{0,1\}$ denote the \textit{indicator random variable} where $\mathbf{1}_{\ce}(\omega)=1$ if $\omega\in \ce$ and $\mathbf{1}_{\ce}(\omega)=0$ otherwise.  

Throughout $\log $ denotes the natural logarithm (base $\mathrm{e}$) we will also use the $\ln$ notation for this natural logarithm sometimes for reader recognition. We note that for any real $x>-1$, 
\begin{equation}\label{logbdd} \frac{x}{1+x}\leq \log(1+x) \leq x.  \end{equation} 
  
We use standard asymptotic (big-$\mathcal{O}$ etc.) notation consistent with \cite{FriezeKaronski}. A sequence of events $(\mathcal{E}_n)$ holds \textit{with high probability} (w.h.p.) if $\lim_{n\rightarrow \infty }\Pr{\mathcal{E}_n}=1$. We use $:=$ to indicate suppressed dependency, e.g.\ $C:=C(c)$ if the constant $C$ depends on $c$.

\subsection{Expansion, Width Measures and Graph Minors} \label{sec:widthdef}

Let $S\subseteq V$ and define the \textit{edge boundary} $ \partial (S)=\{uv\in E(G) : u\in S,v\in V(G)\setminus S\}$ of $S$ to be the set of edges with one endpoint in $S$ and the other outside $S$. Similarly we let the \textit{outer vertex boundary} $\bar{N}(S) = \{ v\in V\backslash S : \text{ there exists }u\in S, uv\in E \}$  of $S$ to be the set of vertices outside $S$ which share an edge with a vertex in $S$. We then define 
\begin{equation*}\phi(G)=\min_{0<|S|\leq n/2} \frac{|\partial(S)|}{|S|} \qquad \text{and}\qquad \iso(G)=\min_{0<|S|\leq n/2} \frac{|\bar{N}(S)|}{|S|}  \end{equation*} to be the \textit{edge} and \textit{vertex isoperimetric numbers} respectively. For any graph $G$ we have 
\begin{equation}\label{eq:expansionrels}
\iso(G) \leq \phi(G)\leq \max_{v\in V} |\bar{N}(v) | \cdot \iso(G).
\end{equation} We say that a graph sequence $G_n$ is an \textit{edge} (resp. \textit{vertex}) \textit{expander sequence} if there exists some fixed $\alpha>0$ such that $ \phi(G_n)\geq \alpha$ (resp. $\iso(G_n)\geq \alpha$) for all $n$ suitably large, see \cite{Exp}.

Let $\tfrac12 \le \alpha < 1$, $s\geq 0$ an integer, and $G=(V,E)$ a graph. A subset $S\subset V$ is said to be an $(s,\alpha)$\textit{-separator} of $G$ \cite{BotBand,Kloks}, if there exists subsets $A,B \subset V$ such that
\begin{itemize}
	\item $V = A \cup B \cup S$ and $A, B, S$ are pairwise disjoint, 
	\item $|S| \leq s$, $|A|, |B| \leq \alpha |V|$, and
	\item $ \{ab\in E : a\in A, b\in B \} =\emptyset$.
\end{itemize}

	A \textit{tree decomposition} of a graph $G=(V,E)$ is a pair $\left(\{X_i: i\in
	I\},\right.$ $\left.T=(I,F)\right)$ where $\{X_i: i \in I\}$ is a family of
	subsets (or `bags') $X_i\subseteq V$ and $T = (I,F)$ is a tree such that
	\begin{itemize}
		\item $\bigcup_{i \in I} X_i = V$,
		\item for every edge $vw \in E$ there exists $i \in I$ with $\{v,w\}
		\subseteq X_i$,
		\item for every $i,j,k \in I$, if $j$ lies on the path
		from $i$ to $k$ in $T$, then $X_i \cap X_k \subseteq X_j$.  
	\end{itemize}
	The \textit{width} of $\left(\{X_i:i \in I\},T=(I,F)\right)$ is defined as
	$\max_{i \in I} |X_i| -1$. The \textit{tree\-width} $\trw(G)$ of $G$ is	the minimum width of any tree decomposition of $G$. Thus, for any graph $G$, $\trw(G)\leq n-1$.
	
For an injective function $f:V(G)\rightarrow \mathbb{Z}$ we define the \textit{cutwidth} \cite{ChungCut} of $G$ by \begin{equation}\label{eq:cwdef} \cw(G)= \min_{f:V\rightarrow \mathbb{Z}, \text{ injective}} \; \max_{x\in \mathbb{R}}\;\left| \left\{ ij \in E(G) : f(i)\leq x <f(j)\right\}\right|;   \end{equation} this is the maximum number of edges crossing a real point when the vertices are arranged in a line according to $f$, minimised over all injections $f:V\rightarrow \mathbb{Z}$.  By \cite[Proposition 1]{GroheM09} and \cite{Bod88}, \begin{equation}\label{trw-sep}	\lfloor \iso(G) \cdot n/4 \rfloor \leq   \trw(G) \leq  \cw(G)\leq |E(G)|. \end{equation} 

A graph $H$ is called a \textit{minor} of the graph $G$ if $H$ can be formed from $G$ by deleting edges and vertices and by contracting edges. 

\begin{lemma}[Folklore, see \cite{Kloks}]\label{sub-div} If $H$ is a minor of $G$ then $ \trw(H)\leq  \trw(G)$.
\end{lemma} 
\subsection{Concentration Inequalities}

Let $\geo{p}$ denote the \textit{geometric distribution} with success probability $p$. That is, if a random variable $X\sim \geo{p}$ then $\Pr{X=k} = (1-p)^{k-1}p  $ for any integer $k\geq 1$. 
\begin{lemma}[{\cite[Theorem 2.3]{JansonTail}}] \label{lem:jansontail}For any $n\geq 1$ and $p_1,\dots  , p_n \in (0,1]$, let $X=\sum_{i=1}^nX_i$ where $X_i \sim \geo{p_i}$. Let $p_*=\min_{i\in [n]}p_i$ and $\mu=\Ex{X}=\sum_{i=1}^n\frac{1}{p_i}$. Then for any $\lambda\geq 1$, \[\Pr{X \geq \lambda \mu  } \leq \lambda^{-1}(1-p_*)^{\mu \left(\lambda -1- \log \lambda\right)}.\]
\end{lemma}

Let $\ber{p}$ denote the Bernoulli distribution with success probability $0\leq p\leq 1$. If $X\sim \ber{p}$ then $\Pr{X=1} = p$ and $\Pr{X=0} = 1- p$. 
 	\begin{lemma}[{\cite[Theorem 4.4]{PandC}}]\label{Chertail}
Let $ n\geq 1$ be an integer, $X=\sum_{i=1}^n X_i$ where $X_i$ are independent Bernoulli random variables, and $\mu=\Ex{X}$. Then, for any real $\delta>0$ we have	\[\Pr{X\geq (1+\delta)\mu} \leq \left(\frac{\mathrm{e}^\delta }{(1+\delta)^{1+\delta} } \right)^{\mu}.\]	
\end{lemma}

\section{The \texorpdfstring{$q$}{q}-Mallows Process \&  Technical Tools for Tangled Paths}\label{sec:mallowtangle}

 In this section we describe a random process which generates a Mallows permutation. It is easier to prove results via this process rather than with $\mu_{n,q}$ directly as the process is driven by independent random input variables. We also introduce a special class of events related to the inputs of this process, and prove a concentration result for sums of indicators of these events. 
\subsection{The \texorpdfstring{$q$}{q}-Mallows Process}\label{sec:Mallowsproc}

In this section we introduce the \emph{$q$-Mallows process} \cite{BhatSubSeq}, a permutation-valued
stochastic process $(r_n)_{n \ge 1}$, where each $r_n \in S_n$. In what follows assume $q > 0$, and $q \neq 1$ (unless specified otherwise). The process is initialized by setting $r_1$ to be the (only) permutation on one element. The process iteratively constructs $r_{n}$ from $r_{n-1}$ and an independent random variable $v_n\sim \nu_{n,q}$, where $\nu_{n,q}$ is the \textit{truncated geometric distribution} given by 
\begin{equation}\label{eq:MallowsDist}
\nu_{n,q}(j)=  \frac{q^{j-1}}{1+q+\cdots+q^{n-1}} =
\frac{(1-q)q^{j-1}}{1-q^n}
\quad(1\le j\le n).
\end{equation} For the case $q=1$ we define $\nu_{n,1}(j) = 1/n$ for all $1\leq j\leq n$. The random sequence $(r_{n})_{n\geq 1}$ can now be defined inductively by sampling  $v_n \sim \nu_{n,q}$ then setting 

\begin{equation}\label{eq:MallowsProc-1}
r_{n}(i) =
\begin{cases}
r_{n-1}(i)& i < v_n\\
n& i=v_n\\
r_{n-1}(i-1) & i > v_n 
\end{cases} \quad (1 \le i \le n). \end{equation}
To visualise this process: start with an empty bookshelf and at each time $i\geq 1$ insert the book with label $i$ at position $v_i - 1$ then shift the remaining books one position to the right.
  
See Figure \ref{fig:Mallowspathexample} for an example of the $q$-Mallows process $(r_n)$ and the resulting tangled path. There is (at least) one other process which generates Mallows permutations, see \cite[Section 2]{BhatSubSeq} for more details. We use $(r_n)$ as it is convenient for our proof methods, and often describe events in terms of the random sequence $(v_n)$  generating the process $(r_n)$. The following lemma tells us that the reversed output permutations $r_n^R$ has the Mallows distribution.

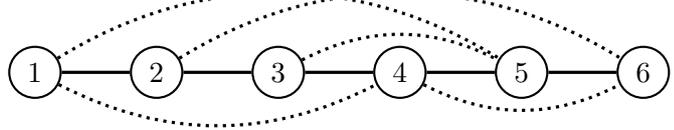
\begin{figure}  
	\center\begin{tikzpicture}[xscale=0.8,yscale=0.7,knoten/.style={thick,circle,draw=black,minimum size=.1cm,fill=white},edge/.style={black,very thick},oedge/.style={dotted,very thick}, bedge/.style={dashed,very thick}]


	\foreach \x in {1,...,6}
	\node[knoten] (\x) at (2*\x,1) {$\x$};
	\draw[edge] (1) to (2);
	\draw[edge] (2) to (3);
	\draw[edge] (3) to (4);
	\draw[edge] (4) to (5);
	\draw[edge] (5) to (6); 
	\draw[oedge] (3) to[out=30,in=150 ] (5); 
	\draw[oedge] (1) to[out=35,in=145 ] (5); 
	\draw[oedge] (1) to[out=-30,in=-150 ] (4);
	\draw[oedge] (4) to[out=-30,in=-150 ] (6);
		\draw[oedge] (2) to[out=35,in=145 ] (6);
			\draw (2.7,2.3) node[anchor=south]{$\layer\left(P_6 ,r_6(P_{6})\right)$};


\draw (-4,-1) node[anchor=south]{$\displaystyle{
\begin{array}{ccl}
n & v_n & r_n\\
1 & 1  & 1\\
2   & 2 & 12\\
3 & 1 &312\\
4 & 3  & 3142\\
5 & 2  & 35142\\
6  & 5   & 351462\\
\end{array}
}$};

	\end{tikzpicture}

	\caption{\label{fig:Mallowspathexample}The table on the left gives the sequences of permutations  $(r_n)$ generated by the sequence $(v_n)$ for $i=1, \dots, 6$. On the right we have a tangled path generated by $r_6$, where the edges of $r_6(P_{6})$ are dotted.}
\end{figure}

\begin{lemma}[{\cite[Corollary 2.3]{BhatSubSeq}}]\label{MallowRelations}Let $q>0$ and $(r_{n})$ be the $q$-Mallows process \eqref{eq:MallowsProc-1}. Then the permutation $\sigma\in S_n$ given by $\sigma(i) = r_{n}(n+1-i)$, for $1\leq i\leq n$, has distribution $\mu_{n,q}$.
\end{lemma}  

 Recall that $\cp(n,q)$ is the tangled path on $n$ vertices with parameter $q$. Any edge in the tangled path $\cp(n,q)$ must belong to one of the two paths $P_n$ or $\sigma(P_n)$, thus \[E(\mathcal{P}(n,q)) = \{ i(i+1); 1\leq i < n \}\cup \{ \sigma(i)\sigma(i+1)\,: \, 1\leq i < n \}.\] However given a permutation $\sigma$, for any $1\leq i\leq n-1$ there exists $1\leq j\leq n-1$ such that $\{\sigma(i),\sigma(i+1) \} = \{\sigma^R(j),\sigma^R(j+1) \} $. It follows that \begin{equation}\label{reversed}\mathcal{P}(n,q) \sim \layer(P_n, \sigma(P)) \sim \layer(P_n, \sigma^R(P_{n})).\end{equation} The equation above shows that it does not matter if we use a $(n,q)$-Mallows permutation or its reverse (which is a $(n,1/q)$-Mallows permutation). This justifies our use of the (unreversed) permutation $r_n$ to generate $\cp(n,q)$ and our restriction of $q$ to the range $0\leq q\leq 1$. 

We now give some useful bounds on densities of the distribution $\nu_{n,q}$. Observe that if $v_k\sim\nu_{k,q}$ then for any integer $1\leq x\leq k$, and $0<q<1$, by \eqref{eq:MallowsDist} we have \begin{equation}\label{eq:nuupperbdd}\Pr{v_k\geq  x}   = \frac{(1-q)\left(q^{x-1}+ \cdots + q^{k-1}\right)}{1- q^{k}}=  \frac{q^{x-1} -q^{k}}{1-q^k}=q^{x-1}\cdot \frac{1 -q^{k-x+1}}{1-q^k}\leq q^{x-1} .\end{equation} 
From this we can also obtain

\begin{equation}\label{eq:Probv_klowerthani}\Pr{v_k\leq x} =1- \Pr{v_k> x} = 1-  \frac{q^x - q^k}{1-q^k}= \frac{1 - q^x}{1-q^k}\geq 1- q^x. \end{equation}

Bhatnagar \& Peled \cite{BhatSubSeq} proved the following tail bounds on the displacement of an element in under the action of a Mallows permutation. 

\begin{theorem}[{\cite[Theorem 1.1]{BhatSubSeq}}]\label{thm:displacement} Let $1\le i\le n$ and $t\ge 1$ be integers, and $0<q<1$ be real. Then, if $\sigma \sim
	\mu_{n,q}$ we have $\Pr{|\sigma(i)-i|\ge t}\le 2q^t.$
\end{theorem}

 \subsection{Local Events and Concentration} \label{sec:klocal}
 
As the $q$-Mallows process $(r_i)$ is generated by a sequence $(v_i)$ of independent random variables,  it is convenient to describe events in terms of $(v_i)$. An example of such an event is the flush \eqref{eq:flushingevent} given by $\mathcal{F}_k= \left\{ \text{for each }i > k\text{ we have }v_i\leq i-k \right\}$. This event will be introduced formally in Section \ref{sec:flushes} and is closely related to the existence of cut vertices, see Section \ref{sec:sep}. In this section we are interested in events defined in terms of the sequence $(v_i)$ which, unlike the flush, become pairwise (asymptotically) independent when the indices of the events are suitably well spaced. To begin we describe a class of events satisfying this criterion; we call these local events. We then show in Lemma \ref{secmom} that sums of indicators of local events concentrate.

 For integers  $n\geq 1$, $i\in [n]$, and $ \ell\geq 1$, we say that an event is \textit{$\ell$-local to $i$} if the event is completely determined by the values of $v_j$ in the range $j \in [i-\ell,i+\ell]$. As an example the event $(\{ v_i\leq 5\}\cup \{v_{i+1} =2\})\cap \{v_{i-2} =3\}  $ is $2$-local to $i$. We say a random variable $X_i$ is $\ell$-local to $i$ if it is a weighted indicator random variable of an event $\ce_{i}$ which is $\ell$-local to $i$, that is $X_i= c_i\cdot \mathbf{1}_{\ce_i}$ where $c_i\geq 0 $ are real numbers. We prove several results by showing that certain small subgraphs are present. Often the number of these subgraphs can be expressed as sums of local random variables, the following lemma, based on Chebyshev's inequality, is then useful.

  \begin{lemma}\label{secmom} Let $n\geq 1$, $\ell \geq 1$ be integers and $S\subseteq [n]$ non-empty. Let $X=\sum_{i\in S} X_i$ where each $X_i$ is a non-negative random variable $\ell$-local to $i$ and $M=\max_{i \in S}\Ex{X_i^2}$. Then, for any  $x\geq 0$,   
 	\[\Pr{\left| X - \Ex{X}\right|> x \cdot \sqrt{M\cdot |S|\cdot \ell} } \leq 5/x^2.  \]  \end{lemma}
 \begin{proof}We seek to apply the second moment method to $X$, and thus we must calculate or bound terms of the form $\Ex{X_i\cdot X_j}$. For any $i,j\in S$ the Cauchy-Schwarz inequality gives \begin{equation}\label{eq:anymombdd}\Ex{X_i\cdot X_j} \leq \sqrt{\Ex{X_i^2}\cdot \Ex{X_j^2}}\leq M. \end{equation} 
 	Observe, each $X_{i}$ only depends on (at most) the values $v_a$ for $a\in [i-\ell,i+\ell]$, and so \begin{equation}\label{eq:farmombdd}\Ex{X_i\cdot X_j} = \Ex{X_i}\cdot \Ex{X_j} \qquad \text{ for any $i,j$ with } |i-j|> 2\ell,  \end{equation} by independence of the sequence $(v_i)_{i=1}^n$. Thus, since $X_i\geq 0$, by \eqref{eq:anymombdd} and \eqref{eq:farmombdd} we have 
 	\begin{align*}	\operatorname{Var}(X) &=  \sum_{i,j\in S}\Ex{X_i\cdot X_j} -\Ex{X_i}\Ex{X_j}\leq   \sum_{i,j\in S} \mathbf{1}_{\{|i-j|\leq 2\ell\}} \cdot  M   \leq |S|(4\ell+1)M\leq 5|S|\ell M.\end{align*}Thus applying Chebyshev's inequality \cite[Theorem 3.6]{PandC}, for any $x\geq 0$, gives\[ \Pr{\left|X-\Ex{X}\right|\geq (x/\sqrt{5})\cdot \sqrt{5|S|\ell M} } \leq \Pr{\left|X-\Ex{X}\right|\geq (x/\sqrt{5})\cdot \sqrt{\operatorname{Var}(X)} } \leq (x/\sqrt{5})^{-2},\] concluding the proof. \end{proof}

\subsection{Flushing Events}\label{sec:flushes}

In this section we define two events with respect to the sequence of values $(v_i)_{i=1}^n$ generating the Mallows process \eqref{eq:MallowsProc-1}. These events will help us describe certain graph properties and prove concentration using Lemma \ref{secmom}. First, we will briefly recall how the $q$-Mallows process $(r_i)_{i=1}^n$ evolves: At step $i\geq 1$ we have a permutation $r_i\in S_i$, starting from $r_1=(1)$. Then, to generate the next permutation $r_{i+1}$ in the sequence we sample a random variable $v_{i+1}\sim \nu_{i+1,q}$ and insert the value $i+1$ at relative position $v_{i+1}\in [1,i+1]$ from the left-hand side of $r_i$. To insert $i+1$ at $v_{i+1}$ we shift each of the values in relative positions $v_{i+1}, \dots , i $ one place to the right.  

 For an integer $ k \in [n]$ we define the \textit{flush} event by
\begin{equation}\label{eq:flushingevent}
\mathcal{F}_k= \left\{ \text{for each }i > k\text{ we have }v_i\leq i-k \right\}, 
\end{equation}
and say there is a \textit{flush} at step $k$ (of the $q$-Mallows process) if $\mathcal{F}_k$ holds (see Figure \ref{fig:flush}).

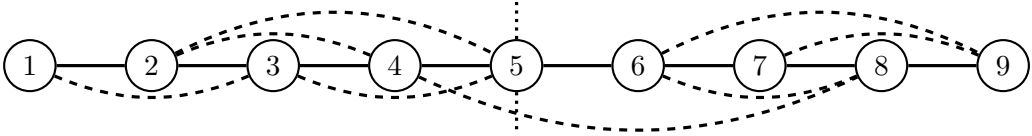
\begin{figure}  
	\center\begin{tikzpicture}[xscale=0.8,yscale=0.7,knoten/.style={thick,circle,draw=black,minimum size=.1cm,fill=white},edge/.style={black,very thick},oedge/.style={orange,very thick}, bedge/.style={dashed,very thick},gedge/.style={green,very thick},medge/.style={magenta,very thick},dotedge/.style={dotted,very thick}]
	\foreach \x in {1,...,9}
	\node[knoten] (\x) at (2*\x,5) {$\x$};
	\node (u) at (10,6.5) {} ;
	\node (d) at (10,3.5) {};
	
	\draw[edge] (1) to (2);
	\draw[edge] (2) to (3);
	\draw[edge] (3) to (4);
	\draw[edge] (4) to (5);
	\draw[edge] (5) to (6); 
	\draw[edge] (6) to (7); 
	\draw[edge] (7) to (8);
	\draw[edge] (8) to (9);
	
	\draw[dotedge] (5) to (u);
	\draw[dotedge] (5) to (d);
	
	\draw[bedge] (7) to[out=25,in=155 ] (9);
	\draw[bedge] (6) to[out=30,in=150 ] (9);
	\draw[bedge] (6) to[out=-25,in=-155 ] (8);
	\draw[bedge] (4) to[out=-28,in=-152 ] (8);
	\draw[bedge] (2) to[out=25,in=155 ] (4);
	\draw[bedge] (2) to[out=30,in=150 ] (5);
	\draw[bedge] (3) to[out=-25,in=-155 ] (5);
	\draw[bedge] (1) to[out=-25,in=-155 ] (3);
 
	\end{tikzpicture}
	\vspace*{-3mm}
	
	\caption{\label{fig:flush}A representation of a flush event $\cf_5$ (see \eqref{eq:flushingevent}) holding in the graph $\layer(P_9,\sigma(P_9))$. In this example $\sigma = (7,9,6,8,4,2,5,3,1)$ was generated by the sequence $\mathbf{x}=(1,1,2,1,3,1,1,3,2)$ satisfying $\cf_5$. Observe that only one edge crosses vertex $5$.}
\end{figure}	

If there is a flush at step $k$ then no subsequent element with value greater than $k$ is inserted into any of the $k$ rightmost positions in the array. Thus, if $\cf_k$ holds, then for any $n\geq k$ the permutation $r_n=sr_k$ will be a concatenation of two strings $s$ and $r_k$, where $s$ is a permutation of the elements $\{k+1, \dots, n\}$ and $r_k$ is the state of the process at step $k$. This property is very useful if we are trying to find certain structures in the permutation $r_n$; since if the structure appears in $r_k$ and $\mathcal{F}_k$ holds (these events are independent) then the elements inserted after $k$ will not affect the structure in $r_k$. One simple example of this is that if $\{v_k=1\}\cap\mathcal{F}_k$ holds then vertex $k$ is a cut vertex of $\cp(n,q)$; see Section \ref{sec:sep} for details and also precise bounds on $\Pr{\mathcal{F}_k}$). Thus, if we wish to show that there is a cut vertex w.h.p., then a standard approach would be to let $X$ count the number of $k\in [n]$ such that $\{v_k=1\}\cap\mathcal{F}_k$ holds, and then bound $\Pr{X=0}$ using the second moment method. This requires control over the variance of $X$. Lemma \ref{secmom} follows this strategy for the special case of local events. However, any event containing the flush event $\mathcal{F}_k$ cannot be $\ell$-local for any $\ell <n-k$ since it specifies that $v_i\leq i-k$ for all $i >k$. 

Observe that, for $q\in(0,1)$, the position each element is inserted is biased towards the left-hand end. Suppose we can find a $b_n:=b_n(q)$ such that  w.h.p.\ for all $i\in [n]$ we have $v_i\leq b_n$. Then, we can define a new event $\mathcal{L}_{k}$ which only specifies the first $b_n$ values in the flush $\mathcal{F}_k$, and if we condition on $\mathcal{L}_k$ then w.h.p.\ $\mathcal{F}_k$ will also hold. To make this precise let
\begin{equation}\label{eq:b}b_n= \left\lceil \frac{8\log n}{\log(1/q)}\right\rceil. \end{equation} Then, as $\Pr{v_k > i}\leq  q^{i} $ for $i<k$ by \eqref{eq:nuupperbdd}, for $n\geq 2$ the union bound gives
\begin{equation}\label{eq:probnobigv}\Pr{\bigcup_{k\in[n]}\{v_k>b_n \} } \leq \sum_{k\in[n]}\Pr{v_k > b_n}=  \sum_{k=b_n+1}^{n} q^{b_n} \leq n\cdot  q^{b_n} \leq  n^{-7}. \end{equation}
If we recall that $v_i$ is the insert position of $i$ relative to the left-hand end, then \eqref{eq:probnobigv} shows that w.h.p.\ no element is inserted at relative position $r>b_n$. The constant $8$ in \eqref{eq:b} is fairly arbitrary but works in our analysis. Thus, conditional on the event $\left\{v_i\leq b_n\text{ for all } i\in [n]\right\}$, if an element $i'<i$ is not within $b_n$ places from the left-hand end of $r_{i-1}$ then when element $i$ is inserted it will not be adjacent to $i'$. This is significant since adjacency of $i$ and $i'$ in the final permutation $r_n$ determines the presence of the edge $ii'$ in $\cp(n,q)$. So, if we want to preserve the subpermutation $r_k$ w.h.p.\ then we do not necessarily need to condition on all elements $i>k$ not being inserted in the rightmost $k$ places, we only need to specify this for the next $b_n$ elements. We can now define, for any $1\leq k\leq n$, the \textit{local flush} event given by \begin{equation}\label{eq:localflushingevent}
\mathcal{L}_k= \left\{\text{ For each } k <i \leq  k+ b_n(q)\text{ we have }v_i\leq i-k \right\}. 
\end{equation} 
The local flush event $\mathcal{L}_k$ captures the desirable property of preserving $r_k$, provided we condition on the event $\ce=\cap_{i\in [n]}\{v_i \leq b_n(q) \}$. As $\ce$ occurs w.h.p.\ by \eqref{eq:probnobigv}, we can essentially use the local flush $\mathcal{L}_k$ in the same way as $\mathcal{F}_k$. The advantage of $\mathcal{L}_k$ is that it only fixes the positions of elements $k+1, \dots, k+b_n$, as opposed to the positions of all elements $i>k$ in $\cf_k$, and thus it is $\lceil b_n/2\rceil$-local with respect to $k + \lceil b_n/2\rceil$. The next result `localises' events involving flushes.

 \begin{lemma}\label{lem:localify} Let $n,\ell\geq 1$  be integers, $0<q<1$, and $(c_i)_{i\in [n]}$ be any non-negative real sequence. Let $\ce_i= \mathcal{B}_i\cap \cf_i  $ where $\mathcal{B}_i$ is $\ell$-local to $i$. Then, the event  $\mathcal{D}_i= \mathcal{B}_i\cap \mathcal{L}_i   $ is $\max\{b_n,\ell\}$-local to $i$, where $b_n:=b_n(q)$ is given by \eqref{eq:b}. Furthermore, if $X=\sum_{i}c_i\cdot \mathbf{1}_{\ce_i}$ and $Y=\sum_{i}c_i\cdot \mathbf{1}_{\mathcal{D}_i}$,  then \[\Pr{X\neq Y}\leq  n^{-7},\qquad \text{ and }\qquad\Ex{X}\leq \Ex{Y}\leq \Ex{ X} + n^{-7}\cdot \sum_{i}c_i   .\]  \end{lemma}
 \begin{proof}
 
 	The event $\mathcal{D}_i$ is $\max\{b_n,\ell\}$-local to $i$ as $\mathcal{B}_i$  is $\ell$-local to $i$ and  $\mathcal{L}_i$, given by \eqref{eq:localflushingevent}, is $b_n$-local to $i$. Recall $v_i \sim\nu_{i,q}$ for $i\in [n]$ and $\mathcal{E}=\cap_{i\in [n]}\{v_i\leq b_n(q)  \}$. Note that $\Pr{\mathcal{E}^c}\leq   n^{-7}$ by \eqref{eq:probnobigv}, and although \eqref{eq:probnobigv} only holds for $n\geq 2$  we can assume this or else $\Pr{\mathcal{E}^c}=0 $ trivially. 
 	
 	Observe that $\ce_i \subseteq \mathcal{D}_i$ for any $i\in [n]$. Furthermore, the events $\ce_i$ and $\mathcal{D}_i$ impose the same restrictions on the random variables $(v_j)_{j\leq i+ b_n(q)}$. In particular, $(v_j)_{i<j\leq i+ b_n(q)}$ satisfy both the flush $\mathcal{F}_i$ and local flush $\mathcal{L}_i$ event. However, conditional on $\mathcal{E}$ no random variable $v_i$ is larger than $b_n(q)$ and so the remaining values $ (v_j)_{j> i + b_n } $ also satisfy $\cf_i$. It follows that $ \mathcal{E}\cap \mathcal{D}_i \subseteq \ce_i$ for any $i\in [n]$ and hence $\Pr{X\neq Y} \leq \Pr{\mathcal{E}^c} $. 
 	
 For the upper bound in the last part of the claim we have $Y=\sum_{i}c_i\cdot \mathbf{1}_{\ce_i}\leq \sum_{i}c_i $ and thus \[\Ex{Y} =  \Ex{X\cdot \mathbf{1}_{\{Y=X\} } } +\Ex{Y \mid Y\neq X }\cdot \Pr{Y\neq X} \leq \Ex{X} +  \left(\sum_{i}c_i \right)\cdot  n^{-7} .\] The lower bound $ \Ex{Y}\geq \Ex{X}$ holds since $\mathcal{D}_i \supseteq \ce_i$ and $c_i\geq 0$ for any $i\in [n]$. 
 \end{proof}

 \section{Expansion and Relations Between \texorpdfstring{$q$}{q}-Mallows Measures}\label{sec:q=1}
 
 This section concerns  $\cp(n,q)$ with $q$ sufficiently close to $1$. In Lemma \ref{lem:expander} we will prove that $\cp(n,1)$ is an expander with probability $1-\mathrm{e}^{-\Omega(n)}$. We then establish  Lemma \ref{lem:largeqpatterncouple} which allows us to relate properties of $\cp(n,q)$ to those of $\cp(n,1)$. This is then used to extend Lemma \ref{lem:expander} to Lemma \ref{treebanddiam1}, which shows that, if $q$ is sufficiently close to $1$, then w.h.p.\ $\cp(n,q)$ is an expander. 
 
 As mentioned in the introduction, it is possible to show that $\cp(n,1)$ is an expander with probability $1-\lo{1}$ by adapting a result of Kim \& Wormald \cite[Theorem 1]{KW}, however it is the fact that Lemma \ref{lem:expander} holds with probability $1-\mathrm{e}^{-\Omega(n)}$ which allows us to relate this to a smaller $q$ in Lemma \ref{treebanddiam1}. This combination of Lemmas \ref{lem:expander} and \ref{lem:largeqpatterncouple} is also used to prove lower bounds on the treewidth in Section \ref{sec:treewidth} by finding a subdivided expander on roughly $1/(1-q)$ vertices as a subgraph in $\cp(n,q)$. Again, if we knew only that the $\cp(n,1)$ was an expander with probability $1-\lo{1}$, then we would not be able to get such good bounds later in Section \ref{sec:treewidth}.

 \subsection{Expansion in the case \texorpdfstring{$q=1$}{q=1}}
 Recall the definition of $ \partial (S)=\{uv\in E(G) : u\in S,v\in V(G)\setminus S\}$, the edge boundary of $S\subseteq V$, from Section \ref{sec:widthdef}.  The following bound is quite crude but (crucially) it is independent of $|S|$.

 \begin{lemma}\label{lem:sizebdd}For any integer $ 1\leq k\leq  n-1$ there are at most $2\cdot \binom{n-1}{k}$ distinct vertex subsets $S$ of an $n$-vertex path such that $|\partial(S)|= k $.
 \end{lemma} 
 \begin{proof}Observe that there are $\binom{n-1}{k}$ ways to choose the $k$ edges of the boundary set $\partial(S)$. Each set $\partial(S)$ of boundary edges gives two possible sets $S$ depending on whether the first vertex of the path is in $S$ or not. 
 \end{proof}

 Recall the definitions of the edge and vertex isoperimetric numbers $\phi$ and $\iso$ from Section \ref{sec:widthdef}. We now prove that w.h.p.\ a tangled path generated from a uniform permutation is an expander.
 
 \begin{lemma}\label{lem:expander}For any integer $n\geq 100$, we have \[\Pr{\iso(\cp(n,1))\leq \frac{1}{40}}\leq 1000  \cdot n^{7/2}\cdot    \left( \frac34\right)^{n}.\]
 \end{lemma}
 
 \begin{proof}Observe that $\iso(\cp(n,1))\geq \phi(\cp(n,1))/4 $ by \eqref{eq:expansionrels} since $\cp(n,1)$ has degree at most $4$. Thus, to prove this Lemma it suffices to bound $\Pr{\phi(\cp(n,1))\leq  1/10}$ from above. 
 	
 	Let $\cp=\layer(P_n, \sigma(P_n))$ where $\sigma\sim \mu_{n,1}$. Now, if the edge isoperimetric number of $\cp$ is at most $\alpha$ then for any set $S\subset [n]$ there can be at most $\alpha|S|$ edges of either the permuted or un-permuted path in $\partial(S)$. That is, for any  $\alpha> 0$ we have  \begin{equation}\label{eq:cupexpression}\{\phi(\cp)\leq    \alpha \} \subseteq \bigcup_{ s =1}^{\lfloor n/2\rfloor}\bigcup_{k=1}^{\lfloor\alpha s \rfloor}\; \; \bigcup_{ S\subseteq V :  |S|=s,  \left| \partial(S)\cap E(P_n)\right| =k  }\{\left| \partial(S)\cap E(\sigma(P_n))\right|\leq \lfloor \alpha s\rfloor -k\}.\end{equation}
 	By Lemma \ref{lem:sizebdd} there are at most $2\cdot\binom{n-1 }{k} $ sets $S$ (of any size) such that  $\left| \partial(S)\cap E(P_n)\right| =k  $. So by applying this bound, the union bound, and rewriting \eqref{eq:cupexpression}, where we ignore the $-k$,  we have  
 	\begin{equation}\label{eq:probexpression} \Pr{\phi(\cp)\leq    \alpha } \leq 2\sum_{s=1}^{\lfloor n/2 \rfloor}\sum_{k=1}^{\lfloor \alpha s\rfloor} \binom{n-1 }{k}   \max_{S\subseteq V : |S|=s}\Pr{\left| \partial(S)\cap E(\sigma(P_n))\right|\leq\lfloor  \alpha s \rfloor  } . \end{equation}  We will now bound the probability on the right-hand side, note that since $\sigma$ is a uniform permutation this is the same for all $S$ with $|S|=s$. 
 	
 	To begin we will view the action of the uniform random permutation $\sigma$ on $P_n$ as a relabelling of $V(P_n) = [n]$. Under this relabelling a given set $S\subseteq V(P_n)$, with $|S|=s$, is equally likely to be mapped to any other $S'\subseteq V(P_n)$ with $|S'|=s$. The number of $S'\subseteq V(P_n)$ with an edge boundary of size $k\geq 0$ is at most $2\cdot \binom{n-1}{ k}$ by Lemma~\ref{lem:sizebdd}. Since $|S'|=s$ there are $s!$ ways to order the elements of $S$ within $S'$ and $(n-s)!$ ways to organise the remaining elements within the path. Thus, if we restrict to a fixed $0<\alpha\leq 1/10$, then for any $S\subseteq V$ with $|S|=s$ we have    
 	\begin{equation}\label{eq:probbdd} \!\!\!\Pr{\left| \partial(S)\cap E(\sigma(P_n))\right|\leq \lfloor  \alpha s \rfloor } \leq 2\sum_{k=1}^{\lfloor \alpha s\rfloor } \binom{n-1}{ k} \frac{s! (n-s)! }{n!} \leq 2n \binom{n-1}{ \lfloor \alpha s\rfloor} \frac{s! (n-s)!}{n!} ,   \end{equation} where the last bound holds since $\alpha s\leq n/20$ and $\binom{x}{y}$ is increasing in $y$ provided $y<\lfloor x/2\rfloor$.    
 	
 	Recall that for any integers $1\leq k\leq n$ and any real number $x$ such that $|x|\leq n$ we have 
 	\begin{equation}\label{eq:basic}\binom{n}{k}\leq \left(\frac{n\mathrm{e}}{k}\right)^k, \qquad  \sqrt{2\pi n}\left(\frac{n}{\mathrm{e}}\right)^n\leq  n!\leq \mathrm{e}\sqrt{n}\left(\frac{n}{\mathrm{e}}\right)^n,\qquad \text{and} \qquad \left(1 + \frac{x}{n} \right)^{n}\leq \mathrm{e}^x.\end{equation}  Thus, as $s\leq \lfloor n/2\rfloor$ and $0<\alpha \leq 1/10$ are fixed, by \eqref{eq:basic} and monotonicity of $\binom{x}{y}$ we have 
 	\begin{equation}\label{eq:binomalpha} \binom{n-1}{\lfloor \alpha s\rfloor } \leq  \left(\frac{n\mathrm{e}}{\lfloor \alpha \lfloor \frac{n}{2}\rfloor \rfloor }\right)^{\lfloor \alpha \lfloor \frac{n}{2}\rfloor \rfloor }\leq   \left( \frac{2\mathrm{e}}{\alpha \left(1-\frac{4}{\alpha n}\right)}\right)^{\frac{\alpha n}{2}}\leq   \left( \mathrm{e}^{-\frac{4}{\alpha}}\right)^{-\frac{\alpha}{2}}\left(\frac{2\mathrm{e}}{\alpha }\right)^{\frac{\alpha n}{2}}=\mathrm{e}^2\left(\frac{2\mathrm{e}}{\alpha }\right)^{\frac{\alpha n}{2}},\end{equation} provided that $n>4/\alpha$. Observe that $s=n/2$ maximises $s^{s}\left( n-s \right)^{n-s}$, thus \eqref{eq:basic} gives\begin{equation}\label{eq:sbddsss}s! \cdot (n-s)! \leq \mathrm{e}\sqrt{s}\left( \frac{s}{\mathrm{e}} \right)^{s} \cdot \mathrm{e}\sqrt{n-s}\left( \frac{n-s}{\mathrm{e}} \right)^{n-s}   \leq \frac{\mathrm{e}^2n}{\mathrm{e}^n} \cdot s^{s}\left( n-s \right)^{n-s} \leq \frac{\mathrm{e}^2n}{\mathrm{e}^n} \cdot \left(\frac{n}{2} \right)^{n}.\end{equation} For $n>4\alpha$, inserting the bounds from \eqref{eq:basic}, \eqref{eq:binomalpha} and \eqref{eq:sbddsss} into \eqref{eq:probexpression} and \eqref{eq:probbdd} gives
 	\begin{align*}
 	\Pr{\phi(\cp)\leq    \alpha } &\leq 2\sum_{s=1}^{\lfloor n/2 \rfloor}\sum_{k=1}^{\lfloor \alpha s\rfloor} \binom{n-1}{k}  \cdot 2n \binom{n-1}{ \lfloor \alpha s\rfloor} \frac{s! (n-s)!}{n!} \\
 	&\leq 2n^2 \mathrm{e}^2\left(\frac{2\mathrm{e}}{\alpha }\right)^{\frac{\alpha n}{2}} \cdot 2n \mathrm{e}^2\left(\frac{2\mathrm{e}}{\alpha }\right)^{\frac{\alpha n}{2}} \frac{\frac{\mathrm{e}^2n}{\mathrm{e}^n} \cdot \left(\frac{n}{2} \right)^{n}}{\sqrt{2\pi n}\left(\frac{n}{\mathrm{e}}\right)^n} \\
 	&= \frac{4\mathrm{e}^6}{\sqrt{2\pi}}\cdot  n^{7/2}\cdot  \left( \frac{2\mathrm{e}}{\alpha}\right)^{\alpha n} 2^{-n}, 
 	\end{align*} One can check that if we fix $\alpha =1/10$ then $\left( \frac{2\mathrm{e}}{\alpha}\right)^{\alpha} \approx 1.4912<3/2$. Thus taking $\alpha=1/10$ gives $\Pr{\phi(\cp)\leq  1/10 } \leq 1000 n^{7/2} (3/4)^n$ for $n\geq 100$ as $4\mathrm{e}^6/\sqrt{2\pi}<1000$.
 \end{proof}

 \subsection{Relating Different \texorpdfstring{$q$}{q}-Mallows Measures}
  The aim of this section is to prove the following result which allows us to relate properties satisfied by the non-uniform tangled path ($q\neq 1$) to those satisfied in the uniform case. Let $2^{\binom{n}{2}}$ denote the set of all (labelled) $n$-vertex graphs.     
  
   \begin{lemma}\label{lem:largeqpatterncouple}Let $n \geq 1$ be an integer, and $0< q<1$. Then, the following holds: \begin{enumerate}[label=(\roman*), left= -6pt]
  		\item\label{itm:c1} For any $A\subseteq S_n$, we have 
  		$\mu_{n,q}(A)\leq \mathrm{e}^{9n^2(1-q)}\cdot \mu_{n,1}(A)  .$
  		\item\label{itm:c2} For any $B\subseteq 2^{\binom{n}{2}}$, we have 
  		$\Pr{\cp(n,q)\in B} \leq \mathrm{e}^{9n^2(1-q)}\cdot \Pr{\cp(n,1)\in B}  .$
  		\item\label{itm:c3} If $q =1 - \lo{n^{-2}}$, then for any $B\subseteq 2^{\binom{n}{2}}$ we have $\Pr{\cp(n,q)\in B} =  \Pr{\cp(n,1)\in B} +\lo{1}.$
  	\end{enumerate} 
  	
  \end{lemma}

  In order to prove this result we must first prove a lemma which bounds the ratio between densities of the truncated geometric and uniform distributions. 
 \begin{lemma}\label{lem:ratio}For any integers $1\leq i\leq k$, and $0<q< 1$, we have
 	\begin{equation*}  \frac{\nu_{k,q}(i)}{\nu_{k,1}(i)} \leq 1 + 9k(1-q) \leq \mathrm{e}^{9k(1-q)} . \end{equation*}
 \end{lemma}
\begin{proof}For ease of notation we will sometimes use the parametrisation $q=1-x$. By the definition \eqref{eq:MallowsDist} of $\nu_{k,q}$, for any $i\leq k$ we have \begin{equation}\label{eq:nuexpression}\nu_{k,q}(i) = \frac{(1-q)q^i}{1-q^k}= \frac{x(1-x)^i}{1-(1-x)^k}.\end{equation}We need a `reverse Bernoulli inequality': for any integer $r\geq 1$ and real $y$ satisfying $|y|\leq 1/(2r)$,
	\begin{equation}\label{eq:revber} (1+y)^r = \sum_{i=0}^r\binom{r}{i}y^i \leq 1 + ry + \sum_{i=2}^r\left(ry\right)^i \leq 1 + ry + (ry)^2\sum_{i=0}^{r-2}  2^{-i} \leq 1 + ry + 2(ry)^2.    \end{equation} We will now begin with the case $1-1/(4k) \leq q <1 $.  Applying \eqref{eq:revber} to \eqref{eq:nuexpression}, for $|x|\leq 1/(4k)$ and $i\leq k$, we have
\begin{equation*}		\nu_{k,q}(i)\leq \frac{x(1-ix+2(ix)^2)}{1 - (1-kx+2(kx)^2)} = \frac{1}{k}\cdot \frac{1-ix+2(ix)^2}{1-2kx} = \frac{1}{k}\cdot\left(1 +\frac{2kx-ix+2(ix)^2}{1-2kx}   \right).
	\end{equation*}
Recall that $\nu_{k,1}(i)=1/k$ for all $1\leq i\leq k$. Thus, by the above and $|x|\leq 1/(4k)$, we have \begin{equation}\label{eq:largeqrat}\frac{\nu_{k,q}(i)}{\nu_{k,1}(i)} \leq 1 +\frac{2kx+2(kx)\cdot (1/4) }{1-2/4 } = 1 + 5kx= 1 + 5k(1-q). 
	\end{equation} 

We now deal with the remaining case $0<q <p $ where $p:= 1-1/(4k) $. By \eqref{eq:nuexpression} we have
\begin{equation}\label{eq:presmallqrat}\frac{\nu_{k,q}(i)}{\nu_{k,p}(i)} =  \frac{(1-q)q^i}{1-q^k}\cdot \frac{1-p^k}{(1-p)p^i} = \frac{1-q}{1-p}\cdot \left(\frac{q}{p}\right)^i \cdot \frac{1-p^k}{1-q^k}\leq  \frac{1-q}{1-p}  = 4k(1-q).    
\end{equation}Thus by \eqref{eq:largeqrat} and \eqref{eq:presmallqrat}, for any $0<q\leq 1 - 1/(4k)$ we have
\begin{equation}\label{eq:smallqrat}
	\frac{\nu_{k,q}(i)}{\nu_{k,1}(i)} = \frac{\nu_{k,q}(i)}{\nu_{k,p}(i)}\cdot \frac{\nu_{k,p}(i)}{\nu_{k,1}(i)} = 4k(1-q)\cdot \left( 1 + 5k \cdot 1/(4k)  \right) =  9k(1-q).
\end{equation}  
The first inequality in the statement follows as $1+9k(1-q)$ is an upper bound for  \eqref{eq:largeqrat}  and \eqref{eq:smallqrat}. Then  the second follows since $ 1+ y \leq \mathrm{e}^y$ for all $y$. \end{proof}
 
Using this lemma we can now prove the main result in this section. 
\begin{proof}[Proof of Lemma \ref{lem:largeqpatterncouple}]By the description of the $q$-Mallows process we see that for every integer $n\geq 1$, $0<q\leq 1$, and $A\subset S_n$, there exists a set $\{(x_k^i)_{k=1}^n : i \in \mathcal{I}\}$ of inputs to the $q$-Mallows process such that 
	\[\mu_{n,q}(A) = \sum_{i\in \mathcal{I}}\prod_{k \in [n]} \nu_{k,q}(x_k^i).\]Furthermore, since the $q$-Mallows is a deterministic function of random inputs, the set $\{(x_k^i)_{k=1}^n : i \in \mathcal{I}\}$ does not depend on $q$. Thus, by Lemma \ref{lem:ratio}, for any event $A\subseteq S_n$ we have 
	\begin{align*}\mu_{n,q}(A) &= \sum_{i\in \mathcal{I}}\prod_{k \in [n]} \nu_{k,q}(x_i^j) \leq   \sum_{i\in \mathcal{I}}\prod_{k \in [n]} \mathrm{e}^{9k(1-q)}\cdot  \nu_{k,1}(x_i^j). \intertext{Then, using the bound $k\leq n$, we can deduce Item \ref{itm:c1} since }\mu_{n,q}(A)&\leq   \mathrm{e}^{9n^2(1-q)}\cdot  \sum_{i\in \mathcal{I}}\prod_{k \in [n]} \nu_{k,1}(x_i^j)= \mathrm{e}^{9n^2(1-q)} \cdot \mu_{n,1}(A).\end{align*} 
		
	\noindent\textit{Item \ref{itm:c2}}: Recall that the tangled path $\cp(n,q)$ is generated deterministically from a Mallows permutation $\sigma\sim\mu_{n,q}$ by the construction $\layer\left(\sigma(P_n),P_n \right)$. Thus for any $B\subseteq 2^{\binom{n}{2}}$ there is a corresponding set of permutations $A\subseteq S_{n}$ such that $\layer\left(\sigma(P_n),P_n \right) \in B$ if and only if $\sigma \in A $. The result now follows from Item \ref{itm:c1}.

	\medskip 
	
	\noindent\textit{Item \ref{itm:c3}}: Since $1-q=\lo{n^{-2}}$, for any $B\subseteq 2^{\binom{n}{2}}$ we have 
	\[\Pr{\cp(n,q)\in B} \leq \mathrm{e}^{9n^2(1-q)}\cdot \Pr{\cp(n,1)\in B} = \mathrm{e}^{o(1)}\cdot \Pr{\cp(n,1)\in B} = \Pr{\cp(n,1)\in B} +o(1), \] by	 Item \ref{itm:c2} and as $\mathrm{e}^{o(1)} = 1+o(1)$ by the Taylor expansion.\end{proof}

 \subsection{Proof of Theorem \ref{treebanddiam}}
 
 We can now apply Lemmas \ref{lem:expander} and \ref{lem:largeqpatterncouple} to prove Lemma \ref{treebanddiam1}. Theorem \ref{treebanddiam} is simply a less explicit restatement of Lemma \ref{treebanddiam1}, so it follows directly from this.

 \begin{lemma}\label{treebanddiam1}	If $ q \geq 1 - \frac{1}{50n}$, then $\Pr{\iso(\cp(n,q))\leq \frac{1}{40}}=\lo{1}$. 
 \end{lemma}
 \begin{proof}Let $B = \{G : \iso(G)\leq 1/40 \}\subseteq 2^{\binom{n}{2}}$. Then Lemmas \ref{lem:largeqpatterncouple}\ref{itm:c2} and \ref{lem:expander} give 
 	\[\Pr{\cp(n,q)\in B}  \leq \mathrm{e}^{9n^2(1-q)}\cdot \Pr{\cp(n,1)\in B}   \leq \mathrm{e}^{n/5} \cdot 1000n^{7/2}\left( \frac34\right)^{n} = o(1),    \] since $\ln(3/4) <-1/4$. 
 \end{proof}

 We note that the constant $1/50$ in the assumption on $q$ in Lemma \ref{treebanddiam1} has not been optimised. However, we believe that $q=1-\Theta(1/n)$ should be the threshold for $\mathcal{P}(n,q)$ being an expander.

\section{Cut Vertices and Diameter}\label{sec:sep}

The main aim of this section is to prove Theorem \ref{prop:sep} which shows that $q= 1- \frac{\pi^2}{6\log n }$ is a sharp threshold for the existence of a cut vertex. This is achieved by equating the existence of a cut vertex to a combination of flush events, then bounding the number of cut vertices. We then use bounds on the number of cut vertices obtained while proving Theorem \ref{prop:sep} to establish Theorem \ref{diam}, which shows that the diameter is linear when $0<q<1$ is fixed.

\subsection{Relating Unit Separators to Flush Events}

Given a graph $G=(V,E)$ we say that a vertex $v$ is a cut vertex in $G$ if its removal separates the graph into two or more disjoint components. The aim of this section is to prove Lemma \ref{lem:seplem} which shows that cut vertices are determined by flush events. 

Let $1\leq k\leq n$ be integers and $v_i\sim\nu_{i,q}$ for all $i\in [n]$. Recall the flush event given by \eqref{eq:flushingevent}: 
\begin{equation*}
\mathcal{F}_k= \left\{\text{For each }i > k\text{ we have }v_i\leq i-k \right\}. 
\end{equation*}
For an integer $ k \in [n]$ we define also define the \textit{reverse flush} event by
\begin{equation*}
\mathcal{R}_k= \left\{ \text{for each }i > k\text{ we have }v_i>k \right\}. 
\end{equation*}
The name we give to this event is quite fitting: Suppose the flush event $\cf_k$ holds in a permutation $\sigma$ and let $\sigma^R$ be the reverse of $\sigma$.  Then, the reverse flush event $\mathcal{R}_k$ holds for $\sigma^R $.

 We define two events $\mathcal{C}_k^{\mathcal{F}}$ and $\mathcal{C}_k^{\mathcal{R}}$ which stipulate that $k$ separates all elements $i>k$ from all elements $i<k$ in $\sigma$. In $\mathcal{C}_k^{\mathcal{F}}$ the elements $i>k$ are to the left of $k$ in $\sigma$, and in $\mathcal{C}_k^{\mathcal{R}}$ they are to the right of $k$ in $\sigma$ (see Figure \ref{fig:separator}). These events are disjoint for $n>1$, and are given by  \begin{equation}\label{eq:E1and2}
\mathcal{C}_k^{\mathcal{F}}=\cf_k  \cap \{ v_k =1 \} \qquad \text{and} \qquad \mathcal{C}_k^{\mathcal{R}}=\mathcal{R}_k  \cap \{ v_k =k \}.
\end{equation} 
Our next result shows that together these two events characterise cut vertices in $\cp(n,q)$.  

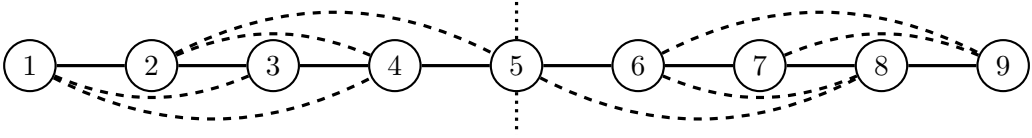
\begin{figure}  
	\center\begin{tikzpicture}[xscale=0.8,yscale=0.7,knoten/.style={thick,circle,draw=black,minimum size=.1cm,fill=white},edge/.style={black, very thick},oedge/.style={orange,very thick}, bedge/.style={dashed,very thick},gedge/.style={green,very thick}, dotedge/.style={dotted,very thick}]
	\foreach \x in {1,...,9}
	\node[knoten] (\x) at (2*\x,5) {$\x$};
	\node (u) at (10,6.5) {} ;
	\node (d) at (10,3.5) {};
	
	\draw[edge] (1) to (2);
	\draw[edge] (2) to (3);
	\draw[edge] (3) to (4);
	\draw[edge] (4) to (5);
	\draw[edge] (5) to (6); 
	\draw[edge] (6) to (7); 
	\draw[edge] (7) to (8);
	\draw[edge] (8) to (9);
	
	\draw[dotedge] (5) to (u);
	\draw[dotedge] (5) to (d);
	
	\draw[bedge] (7) to[out=25,in=155 ] (9);
	\draw[bedge] (6) to[out=30,in=150 ] (9);
	\draw[bedge] (6) to[out=-25,in=-155 ] (8);
	\draw[bedge] (5) to[out=-30,in=-150 ] (8);
	\draw[bedge] (2) to[out=30,in=150 ] (5);
	\draw[bedge] (2) to[out=25,in=155 ] (4);
	\draw[bedge] (1) to[out=-30,in=-150 ] (4);
	\draw[bedge] (1) to[out=-25,in=-155 ] (3);
 
	\end{tikzpicture}
	\vspace*{-3mm}
	
	\caption{\label{fig:separator}A representation of the event $\mathcal{C}_5^\cf$ from \eqref{eq:E1and2} holding in the graph $\layer(P_9,\sigma(P_9))$. In this example $\sigma = (7,9,6,8,5,2,4,1,3)$ was generated by the sequence $\mathbf{x}=(1,1,3,2,1,1,1,3,2)$ satisfying $\mathcal{C}_5^\cf$. Observe that vertex $5$ is a cut vertex.}
\end{figure}	

	\begin{lemma}\label{lem:seplem} Let $n\geq 3$ be an integer. Then, a vertex $k \in \{2,\dots,n-1\}$ is a cut vertex in $\mathcal{P}(n,q)$, if and only if one of the disjoint events $\mathcal{C}_k^{\mathcal{F}}$ or $\mathcal{C}_k^{\mathcal{R}}$ holds.  
\end{lemma}

\begin{proof}We consider the tangled path $\layer(P_n,r_n(P_n))$ with $r_n\sim \mu_{n,q}$ generated from the random sequence $(v_i)_{i=1}^n$, where $v_i\sim\nu_{i,q}$, via the $q$-Mallows process. Recall from \eqref{eq:MallowsProc-1} that given $(v_i)_{i=1}^n$ we generate the sequence $(r_i)_{i=1}^n$ inductively by the rule
	\begin{equation*}
	r_{i}(j) =
	\begin{cases}
	r_{i-1}(j)& j < v_i\\
	i& j=v_i\\
	r_{i-1}(j-1) & i > v_j 
	\end{cases} \qquad (1 \le j \le i). \end{equation*}
	
 First assume that $\mathcal{C}_k^{\mathcal{F}}$ holds for some $k \in \{2,\dots,n-1\}$. Conditional on $\{ v_k =1 \}$ we have $r_k(1)=k$ and $r_k(i)< k$ for all $i>1$. Once $k$ is inserted then, conditional on $\cf_k$, each $i>k$ the element $i$ is inserted to the left of $k$ in the permutation $(r_{i-1})$. It follows that, conditional on $\mathcal{C}_k^{\mathcal{F}}$, the final permutation $r_n$ satisfies $r_k(i)> k$ for all $i\leq n-k$, $r_n(n-k+1)=k$, and $r_n(j)<k$ for all $j>n-k+1$. Thus, since no element $i>k$ is adjacent to any element $j<k$ in $r_n$ and thus $k$ is a cut vertex of $\cp$. See Figure \ref{fig:separator} for an example.  We now assume $\mathcal{C}_k^{\mathcal{R}}$ holds. Conditional on $\{ v_k =k \}$ we have $r_k(k)=k$ and $r_k(i)< k$ for all $i<k$. Then, conditioning on $\mathcal{R}_k$, each element $i>k$ is inserted to the right of $k$. So, conditional on $\mathcal{C}_k^{\mathcal{R}}$, we have $r_k(i)< k$ for all $i<k$, $r_n(k)=k$ and $r_n(j)>k$ for all $j>k$. Thus, as before, $k$ is a cut vertex of $\cp$. 
	
 For the other direction suppose $k \in \{2,\dots,n-1\}$ is a cut vertex. Then, since $\cp$ contains a path on $[n]$ as a subgraph, $k$'s removal must separate the graph into vertex sets $\{1, \dots, k-1 \}$ and $\{k+1, \dots, n \}$ and there can be no element $j>k$ adjacent to any $i<k$ in $r_n$. By the definition of $(r_i)_{i=1}^n$, at the time $k$ is inserted all elements $i< k$ have been inserted. If $k$ is a cut then either $\{v_k=1 \}$ or $\{v_k=k\}$ holds, since otherwise when $k+1$ is inserted it will be adjacent to at least one element $j<k$. Now if $\{v_k=1 \}$ holds then all elements $i>k$ must be inserted to the left of $k$, this is precisely the event $\cf_k $. Otherwise if $\{v_k=k\}$ holds then all $i>k$ must be inserted right of $k$, this gives $\mathcal{R}_k$. Hence, if $k$ is a cut then either $\mathcal{C}_k^{\mathcal{F}}$ or $\mathcal{C}_k^{\mathcal{R}}$ holds.\end{proof}

\subsection{Bounds on the Probability of Flush Events}\label{sec:flusprobs}

 Our first result in this section gives algebraic expressions for $\Pr{\mathcal{F}_k}$ and $\Pr{\mathcal{R}_k} $.

\begin{lemma}\label{lem:expforF_k}For integers $1\leq k\leq n$, and $0<q< 1$, we have \[\Pr{\mathcal{F}_k}= \prod_{i=1}^{k}\frac{1-q^{i}}{1-q^{n-k+i}}\qquad \text{and} \qquad \Pr{\mathcal{R}_k}=  q^{k(n-k)}\cdot \Pr{\cf_k}.\]
\end{lemma}

\begin{proof}Observe that for each $i>k$ we have $\Pr{v_i\leq i- k}   = \frac{1-q^{i-k}}{1-q^i}$ by \eqref{eq:Probv_klowerthani}, thus 
	\begin{equation} \label{eq:firstprobespression} \Pr{\cf_k} = \prod_{i=k+1}^n\Pr{v_i\leq i- k}=\prod_{i=k+1}^{n} \frac{1-q^{i-k}}{1-q^i}=\left(\prod_{i=k+1}^{n} (1-q^i)\right)^{-1}\cdot \prod_{i=k+1}^{n}\left(1-q^{i-k}\right)  ,\end{equation}where the first equality holds since $(v_i)$ are independent. By shifting some indices we have 	\begin{equation*} \Pr{\cf_k}=\left(\prod_{i=1}^{k} (1-q^{n-k +i}) \cdot \prod_{i=k+1}^{n-k} (1-q^i)\right)^{-1}\cdot \prod_{i=1}^{n-k}\left(1-q^{i}\right)  = \prod_{i=1}^{k}\frac{1-q^{i}}{1-q^{n-k+i}},\end{equation*}as claimed. Now, again by \eqref{eq:nuupperbdd}, we have $\Pr{v_i> k}   = \frac{q^{k}-q^i}{1-q^i}$ for any $i>k$ and thus 	\begin{equation*}  \Pr{\mathcal{R}_k} = \prod_{i=k+1}^n\Pr{v_i> k} = \prod_{i=k+1}^{n} \frac{q^{k}-q^i}{1-q^i} = q^{k(n-k)}\prod_{i=k+1}^{n} \frac{1-q^{i-k}}{1-q^i} = q^{k(n-k)}\cdot \Pr{\cf_k}, \end{equation*} where the last equality follows from the second inequality of \eqref{eq:firstprobespression}. \end{proof}

The following technical lemma is used for obtaining  tight bounds on $\Pr{\cf_k}$. 
\begin{lemma}\label{lem:maclaurin}For any $0<q<1$, we have
	\[\frac{q\log q}{6(1-q)}\leq \sum_{i=1}^\infty \log(1-q^i) - \frac{\pi^2}{6\log q} - \frac{3\log (1-q)}{2}\leq   - \frac{1-q}{q\log q}.  \]
\end{lemma}
\begin{proof}Recall the Euler-Maclaurin summation formula \cite[Section 9.5]{concrete}: for any real function $f(x)$ with derivative $f'(x)$ and any integers $-\infty< a<b<\infty $ we have 
	\begin{equation}\label{eq:EM}\sum_{a\leq i <b} f(i)  -\int_{a}^bf(x)\, \mathrm{d}x +\frac{1}{2} f(x)\Big|_{a}^b   =  \frac{1}{12}  f'(x)\Big|_{a}^b + R_{2},  \end{equation} where $f(x)\big|_{a}^b = f(b) -f(a)$ and $R_{2}$ is real and satisfies $|R_{2}|\leq \frac{1}{12} \left| f'(x)\big|_{a}^b\right|$.  
	
	In our case we set $f(x)= \log(1-q^x)$, thus $\frac{1}{2} f(x)\Big|_{a}^b = - \frac{1}{2}\log\left(\frac{1-q^a}{1-q^b}\right)$. The derivative of $f$ is \[f'(x) = -q^x\log(q)/(1-q^x) = \log(1/q)\left(1/(1-q^x) - 1 \right)> 0.\] Notice that $f'(x)$ is decreasing in $x$, thus as $f'(a)>f'(b)$ since $b>a$, we have  \[  0\geq  \frac{1}{12}  f'(x)\Big|_{a}^b + R_{2}\geq  \frac{1}{6}f'(x)\Big|_{a}^b = \frac{q^a\log q}{6(1-q^a)} -  \frac{q^b\log q}{6(1-q^b)}.    \]So by \eqref{eq:EM} and setting $a=1$, for any $b\geq 1$ we have
	\begin{align*}   \frac{q\log q}{6(1-q)} - \frac{q^b\log q}{6(1-q^b)} \leq \sum_{1\leq i <b} \log(1-q^i) - \int_{1}^b  \log(1-q^x)\,\mathrm{d}x - \frac{1}{2}\log \left(\frac{1-q}{1-q^b}\right)\leq  0. \end{align*}   
	Now if we take $b\rightarrow \infty$ then the integral and sum converge since $q<1$. Thus, we have
	\begin{align}\label{eq:firstlogbdd} \frac{q\log q}{6(1-q)} \leq \sum_{i=1}^\infty  \log(1-q^i) - \int_{1}^\infty  \log(1-q^x)\,\mathrm{d}x - \frac{\log (1-q)}{2}\leq  0. \end{align}   
Observe that the substitution $y= q^x$ yields
	\begin{equation}\label{eq:intbdd} \int_1^\infty \log(1-q^x) \, \mathrm{d} x = \frac{1}{\log q} \int_0^{q} \frac{\log(1-y)}{y} \, \mathrm{d} y. \end{equation}
	The right-hand side of \eqref{eq:intbdd} contains the  Dilogarithm function $\operatorname{Li}_2(q)$ given by \begin{equation}\label{eq:dilog}\operatorname{Li}_2(q) := \int_0^{q} \frac{\log(1-y)}{y} \, \mathrm{d} y = \sum_{j=1}^\infty\frac{q^j}{j^2},\qquad \text{for $0\leq q\leq 1$},\end{equation} where the second equality is by \cite[(1.3)]{Lewin}, thus  $\operatorname{Li}_2(1) = \pi^2/6$. Integration by parts gives 
	\begin{equation*}
			\int_0^{q} \frac{\log(1-y)}{y} \, \mathrm{d} y = \ln(y)\ln(1-y)\big|_0^q + 	\int_0^{q} \frac{\log(y)}{1-y} \, \mathrm{d} y =\ln(q)\ln(1-q) +  \int_{1-q}^{1} \frac{\log(1-x)}{x} \, \mathrm{d} x, 
		\end{equation*}where we used the substitution $x=1-y$. By \eqref{eq:dilog} and the above we have
	\begin{equation*}
	\int_0^{q} \frac{\log(1-y)}{y} \, \mathrm{d} y =\ln(q)\ln(1-q) +  \operatorname{Li}_2(1) -\operatorname{Li}_2(1-q)  =\ln(q)\ln(1-q) + \frac{\pi^2}{6}- \sum_{j=1}^\infty\frac{(1-q)^j}{j^2}. 
\end{equation*}	Note that $0\leq \sum_{j=1}^\infty\frac{(1-q)^j}{j^2} \leq\sum_{j=1}^\infty (1-q)^j= \frac{1-q}{q}$. The result then follows from \eqref{eq:firstlogbdd}.
\end{proof}

We can now apply this approximation to prove bounds on  $\Pr{\cf_k}$. 
\begin{lemma}\label{lem:flushlem}For any integers $1\leq k\leq n$, and $0<q< 1$,  we have \[ \frac{q\log q}{6(1-q)} \leq \log \Pr{\cf_k} - \frac{\pi^2}{6\log q} - \frac{3\log (1-q)}{2} \leq   \frac{2q^{\min\{n-k, k\}}}{(1-q)(1-q^{\min\{n-k, k\}})}  - \frac{1-q}{q\log q}.\] 
\end{lemma}

\begin{proof}  By Lemma \ref{lem:expforF_k} we have $\Pr{\mathcal{F}_k}= \prod_{i=1}^{k}\frac{1-q^{i}}{1-q^{n-k+i}}$ for any $1\leq k\leq n$. Thus	\begin{equation}\label{eq:logP} \begin{aligned} \log \Pr{\cf_k} &=  \sum_{i=1}^{k}\log \left(1-q^i \right)  - \sum_{i=1}^{k}\log\left(1-q^{n-k+i} \right) \\ &=  \sum_{i=1}^{\infty} \log \left(1-q^i \right)  - \sum_{i=1}^{\infty} \log \left(1-q^{k +i } \right)  -\sum_{i=1}^{k} \log\left(1-q^{n-k+i} \right). \end{aligned}\end{equation} Now, since $\log(1-x)  \geq - x/(1-x) $ for all $x<1$, we have \[0 \leq - \sum_{i=1}^{k}\log\left(1-q^{n-k+i} \right)\leq \sum_{i=1}^{k}\frac{q^{n-k+i}}{1-q^{n-k+i}}\leq \frac{1}{1-q^{n-k+1}}\sum_{i=1}^{k}q^{n-k+i}\leq \frac{q^{n-k+1}}{(1-q)(1-q^{n-k+1})},  \] and similarly $0\leq - \sum_{i=1}^{\infty} \log \left(1-q^{k +i } \right)\leq \frac{q^{k+1}}{(1-q)(1-q^{k+1})}  $.  Thus 
	\begin{equation*}
	0\leq -\sum_{i=1}^{k}\log\left(1-q^{n-k+i} \right) - \sum_{i=1}^{\infty} \log \left(1-q^{k +i } \right) \leq \frac{2q^{\min\{n-k, k\}}}{(1-q)(1-q^{\min\{n-k, k\}})} .
	\end{equation*}	The result now follows by inserting these bounds into \eqref{eq:logP} and then using Lemma \ref{lem:maclaurin} to bound the remaining sum in \eqref{eq:logP}.
\end{proof}
The following lemma gives a better bound on $\Pr{\cf_k}$ than Lemma \ref{lem:flushlem} when $q$ is close to $1$.

\begin{lemma}\label{lem:cheapflush}For any integers $1\leq k\leq n$,  and $0<q< 1$, we have \[\Pr{\cf_k}\leq \exp\left( -  \frac{q \left(1-q^{\min\{k,n-k \}} \right)}{2(1-q)}\right).\]
\end{lemma}
\begin{proof} By \eqref{eq:Probv_klowerthani}, $\Pr{v_i\leq i- k}   =  1- \frac{q^{i-k}-q^i}{1-q^i}$. Independence of the random variables $v_i$ yields	\begin{align*}\Pr{\cf_k} &= \prod_{i=k+1}^n\Pr{v_i\leq i- k}=\prod_{i=1}^{n-k} \left( 1- \frac{q^{i}-q^{k+i}}{1-q^{k+i}}\right)\leq \prod_{i=1}^{n-k} \exp\left(- \frac{q^{i}-q^{k+i}}{1-q^{k+i}}\right),\end{align*}where the last inequality follows since $ 1+x\leq \mathrm{e}^x$ for all real $x$. This gives the following, where in the last step we apply the expression for the sum of a geometric series:\begin{align*}  \Pr{\cf_k} &\leq \exp\left(- \sum_{i=1}^{n-k} \frac{q^{i}-q^{k+i}}{1-q^{k+i}}  \right)\leq \exp\left(- \frac{1-q^{k}}{1-q^{n}} \sum_{i=1}^{n-k}q^{i} \right)=\exp\left(- \frac{q(1-q^{k})(1-q^{n-k})}{(1-q)(1-q^{n})} \right).\end{align*} Observe that $1-q^n = (1+q^{ n/2 })(1-q^{  n/2 })$ and $\max\{k,n-k\}\geq  n/2   $, thus  \begin{align*} 
	\Pr{\cf_k}	&\leq \exp\left(- \frac{q}{1-q} \cdot \frac{(1-q^{k})(1-q^{n-k})}{(1+q^{ n/2 })(1-q^{  n/2 })}\right) \leq\exp\left(- \frac{q}{1-q}\cdot  \frac{1-q^{\min\{n-k,k\}}}{1+q^{ n/2 }} \right),
	\end{align*} and finally $\Pr{\cf_k} \leq \exp\left(- \frac{q(1-q^{\min\{n-k,k\}})}{2(1-q)}\right)$ since $1+q^{ n/2  }\leq 2$. 
\end{proof}

\subsection{Bounds on the Number of Cuts}\label{sec:cutcount}

For $1/2< \alpha <1$ we let the random variable $X_n(\alpha)$ denote the number of vertices $k\in [\lceil (1-\alpha)n\rceil,\lfloor \alpha n\rfloor   ]$ which are cut vertices in $\cp(n,q)$. We have the following by Lemma \ref{lem:seplem}.

\begin{corollary}\label{cor:sepcor} Let $n\geq 1$ be an integer,  $1/2<\alpha <1$, and $0\leq q\leq 1$. Then \[X_n(\alpha) =\sum_{k = \lceil (1-\alpha)n\rceil}^{\lfloor \alpha n\rfloor} \left(\mathbf{1}_{\mathcal{C}_{k}^{\mathcal{F}}}+ \mathbf{1}_{\mathcal{C}_{k}^{\mathcal{R}}}\right).\]
\end{corollary}
Restricting the ranges of $q$ and $k$ in Lemma \ref{lem:flushlem} gives the following result. 
\begin{lemma}\label{lem:precisecor}Let $1/2 < \alpha < 1$, $n \geq  ( \frac{100}{1-\alpha}
	)^5$ be an integer, 
	and $0 < q \leq  1 -n^{4/5}$. Then, for any integer $ k\in [\lceil (1-\alpha)n\rceil,\lfloor \alpha n\rfloor   ]$, we have  
	\[ \mathrm{e}^{-1/6} \cdot (1-q)^{5/2}\cdot  \exp\left(-\frac{\pi^2}{6(1-q)}\right) \leq \Pr{\mathcal{C}_k^{\mathcal{F}}} \leq \mathrm{e}^{6} \cdot (1-q)^{5/2}\cdot \exp\left(-\frac{\pi^2}{6(1-q)}\right).\]
\end{lemma}  

\begin{proof}Recall $\Pr{ v_k=1 } = \frac{1-q}{1-q^k}\geq 1-q$ from \eqref{eq:MallowsDist}. Since $ k\in \big[\lceil (1-\alpha)n\rceil,\lfloor \alpha n\rfloor   \big]$, we have \begin{equation}\label{eq:bddonq^k}q^k \leq q^{\min\{k , n-k \}}\leq  (1-1/n^{4/5})^{(1-\alpha) n} \leq (\mathrm{e}^{-1/n^{4/5}})^{(1-\alpha) n} \leq \mathrm{e}^{-(1-\alpha) n^{1/5}}\leq \mathrm{e}^{-100},\end{equation}where the last step uses $n\geq (\frac{100}{1-\alpha})^5$. By independence, $\Pr{\mathcal{C}_k^{\mathcal{F}}} =  \Pr{\cf_k}\cdot \Pr{ v_k=1 }$ and so  
	\begin{equation}\label{eq:approx}(1-q) \cdot \Pr{\cf_k}  \leq \Pr{\mathcal{C}_k^{\mathcal{F}}} = \frac{1-q}{1-q^k}\cdot\Pr{\cf_k} \leq \mathrm{e} \cdot (1-q)\cdot\Pr{\cf_k}. \end{equation} By \eqref{eq:approx} it suffices to bound $\Pr{\cf_k}$. Taking $x=q-1>-1$, the bounds on $\log$ from \eqref{logbdd} give \begin{equation}\label{eq:logbdd10}\frac{1}{\log q} \geq -\frac{1}{1-q}, \quad \log q \geq  -\frac{1-q}{q}, \quad \text{and}\quad \frac{1}{\log q} \leq 1 -\frac{1}{1-q}   .\end{equation} Thus, applying \eqref{eq:logbdd10} to the bound on $\Pr{\cf_k}$ in Lemma \ref{lem:flushlem} gives 
	\begin{equation}\label{eq:downbdd10} \log \Pr{\cf_k}  \geq \frac{\pi^2}{6\log q} + \frac{3\log (1-q)}{2} +  \frac{q\log q}{6(1-q)}\geq -\frac{\pi^2}{6(1-q)} + \frac{3\log (1-q)}{2} - \frac{1 }{6 }.\end{equation}Inserting \eqref{eq:downbdd10} into \eqref{eq:approx} gives the lower bound in the statement.
	
	For the upper bound applying  \eqref{eq:logbdd10} and \eqref{eq:bddonq^k} to Lemma \ref{lem:flushlem} gives 
	\begin{align} \log \Pr{\cf_k} +\frac{\pi^2}{6(1-q)} -\frac{3\log (1-q)}{2}&\leq  \frac{\pi^2}{6}  +   \frac{2q^{\min\{k , n-k \}}}{(1-q)(1-q^{\min\{k , n-k \}})}  - \frac{1-q}{q\log q}\notag \\  &\leq  \frac{\pi^2}{6} + 3n^{4/5}\mathrm{e}^{-n^{1/5}/4} + \frac{1}{q}\notag\\
	&\leq 2 + \frac{1}{q}.\label{eq:upbdd10}\end{align} Let $f(q) = \mathrm{e}^{3+1/q}  \cdot (1-q)^{5/2}\cdot \mathrm{e}^{-\frac{\pi^2}{6(1-q)}}  $ and observe that applying \eqref{eq:upbdd10} to \eqref{eq:approx} gives the bound $\Pr{\mathcal{C}_k^{\mathcal{F}}} \leq f(q)$. Now, as $f(q)$ is monotone decreasing in $q$, for any $q\leq 1/3$ we have $f(q)\geq f(1/3) = \mathrm{e}^{6}  \mathrm{e}^{-\frac{\pi^2}{4}}(2/3)^{5/2}  >1 $. Thus since $\Pr{\mathcal{C}_k^{\mathcal{F}}}\leq \max\{f(q),1 \}$ holds and $\mathrm{e}^{3+1/q}$ is monotone decreasing in $q$ we can simplify the bound to $ \Pr{\mathcal{C}_k^{\mathcal{F}}} \leq\mathrm{e}^{6} \cdot (1-q)^{5/2}\cdot  \mathrm{e}^{-\frac{\pi^2}{6(1-q)}} $, as claimed. \end{proof}

The next lemma shows $\mathcal{C}_k^{\mathcal{R}}$ is `rarer' than $\mathcal{C}_k^{\mathcal{F}}$, except if $q=1$ where they are equiprobable. 
\begin{lemma}\label{lem:R_kbyF_k} For integers $1\leq k\leq n$, and $0<q< 1$, we have $ \Pr{\mathcal{C}_k^{\mathcal{R}}}=  q^{k(n-k+1)-1}  \Pr{\mathcal{C}_k^{\mathcal{F}}}.$
\end{lemma}
\begin{proof}Recall that $\Pr{v_k=j} = \frac{(1-q)q^{j-1}}{1-q^k} $ by \eqref{eq:MallowsDist}, and so $\Pr{v_k=k}= q^{k-1}\Pr{v_k=1}$. The result follows from independence and since $\Pr{\mathcal{R}_k}=  q^{k(n-k)}\cdot \Pr{\cf_k}  $ by Lemma \ref{lem:expforF_k}.    
\end{proof}

Lemma \ref{lem:lowerX_n} below shows that for a certain range of $q$ there are many cut vertices. This proves one side of Theorem \ref{prop:sep} and will also be key to the proof of Theorem \ref{diam}.

\begin{lemma}\label{lem:lowerX_n}Let $1/2<\alpha <1$, $n\geq \big(\frac{100}{\min\{1-\alpha, \;2\alpha -1\}}\big)^5$ be an integer, and $0\leq q\leq 1-n^{-4/5}$.  Then, for any $x> 0$,	\[  \quad  \Pr{ X_n(\alpha)  < \frac{ 2\alpha -1  }{10} \cdot n (1-q)^{5/2}  \mathrm{e}^{-\frac{\pi^2}{6(1-q)}}\left(1  -     \frac{1000x}{\sqrt{2\alpha -1}}\cdot \sqrt{\frac{\mathrm{e}^{\frac{\pi^2}{6(1-q)}} \log  n }{n(1-q)^{7/2} } }  \right)}\leq \frac{5}{x^2}+ \frac{1}{n^7}.\]
\end{lemma}
\begin{proof} To begin, by Corollary \ref{cor:sepcor} and Lemma \ref{lem:precisecor},
\[ 	\sum_{k =\lceil (1-\alpha)n\rceil}^{\lfloor \alpha n\rfloor}  \Pr{\mathcal{C}_k^{\mathcal{F}}}
\geq ((2\alpha -1)n-2)\cdot \mathrm{e}^{-1/6}   (1-q)^{5/2}  \mathrm{e}^{-\frac{\pi^2}{6(1-q)}}.\] Since $n$ satisfies $n\geq\big(\frac{100}{\min\{1-\alpha, \;2\alpha -1\}}\big)^5$, we have  \begin{equation}\label{eq:lowerESn}\Ex{X_n(\alpha)}
	\geq \frac{ 2\alpha -1  }{10} \cdot n\cdot (1-q)^{5/2}\cdot  \mathrm{e}^{-\frac{\pi^2}{6(1-q)}} .\end{equation}  

We want to show that  $X_n(\alpha)$ concentrates around $\Ex{X_n(\alpha)}$, but we cannot apply Lemma \ref{secmom} directly as $\mathcal{C}_k^{\mathcal{F}}$ is not $\ell$-local for any reasonable $\ell$. Recalling $\mathcal{L}_k= \left\{v_i\leq i-k \text{ for all } k <i \leq  k+ b_n\right\}$ from \eqref{eq:localflushingevent}, where $b_n= \left\lceil \frac{8\log n }{\log(1/q)}\right\rceil$  by \eqref{eq:b}, define the event $\mathcal{C}_k^{\mathcal{L}}=\mathcal{L}_k\cap \{v_k= 1\}$, and let\[ Y_n(\alpha)= \sum_{k =\lceil (1-\alpha)n\rceil}^{\lfloor \alpha n\rfloor}  \mathbf{1}_{ \mathcal{C}_k^{\mathcal{L}}} .\]  

We note two facts that follow from the first part of the statement in Lemma \ref{lem:localify}, the bound $\log(1/q) \geq 1-q$, and the assumption $n\geq\big(\frac{100}{\min\{1-\alpha, \;2\alpha -1\}}\big)^5$. Firstly,  $\mathcal{C}_k^{\mathcal{L}}$ is $\ell$-local to $k+ \lceil \frac{b}{2} \rceil$, where $\ell=\lceil\frac{1+b_n}{2}\rceil  \leq \frac{5\log n }{1-q} $. Secondly, $\Pr{\mathcal{C}_k^{\mathcal{F}}}\leq \Pr{\mathcal{C}_k^{\mathcal{L}}}\leq  (1+n^{-7})\Pr{\mathcal{C}_k^{\mathcal{F}}}$.  It follows that $ Y(\alpha)$ is the `localisation' of $X(\alpha)$ according to Lemma \ref{lem:localify}, where $c_i=1$ for all $i\in[n]$, also \begin{equation}\label{eq:ExXY}\Ex{X_n(\alpha)} = \sum_{k =\lceil (1-\alpha)n\rceil}^{\lfloor \alpha n\rfloor}   \Pr{\mathcal{C}_k^{\mathcal{F}}}\leq \Ex{Y_n(\alpha)} \leq (1+n^{-7})\cdot \Ex{X_n(\alpha)} \leq  2\cdot \Ex{X_n(\alpha)}.\end{equation} Recall $S$ and $M$ from Lemma \ref{secmom}. In this setting $S= \big[\lceil (1-\alpha)n\rceil,\lfloor \alpha n\rfloor   \big]$ and, by Lemma \ref{lem:precisecor}, \[M=\max_{k \in S}\Pr{\mathcal{C}_k^{\mathcal{L}}} \leq (1+n^{-7})\cdot \max_{k \in S}\Pr{\mathcal{C}_k^{\mathcal{F}}}\leq 2\cdot \mathrm{e}^{6} (1-q)^{5/2}\mathrm{e}^{-\frac{\pi^2}{6(1-q)}}. \]
Since $10^2\mathrm{e}^{6}\left(1-\frac{2}{(2\alpha -1)n}\right)^{-1} \leq 250^2 $, we have the following by comparison with \eqref{eq:lowerESn}:
	\begin{equation}\label{eq:MLnbdd}M\cdot |S|\cdot \ell\leq  2 \mathrm{e}^{6}(1-q)^{5/2}\mathrm{e}^{-\frac{\pi^2}{6(1-q)}}  \cdot \left( 2\alpha-1\right) n \cdot \frac{5\log n }{1-q}  \leq \frac{250^2\log n }{ 1 -q}\cdot  \Ex{X_n(\alpha)}.\end{equation}
	Finally, for any $x>0$, by Lemma \ref{secmom}, \eqref{eq:ExXY}, and \eqref{eq:MLnbdd}\begin{align}\label{eq:conccuts}\Pr{ Y_n(\alpha) < \frac{\Ex{ X_n(\alpha) }}{2}  -    \sqrt{\frac{(250x)^2\log n }{1-q}\cdot \Ex{  X_n(\alpha) } }  }  &\leq \Pr{ Y_n(\alpha) <  \Ex{Y_n(\alpha)} -  x \cdot \sqrt{M\cdot |S|\cdot \ell} }\notag\\
		&\leq 5/x^2 . \end{align} Now,  by \eqref{eq:lowerESn}, we have	 \[   \frac{(250x)^2\cdot \log n }{\Ex{  X_n(\alpha) }\cdot (1-q)} \leq  \frac{(250x)^2\cdot \log n}{ \frac{ 2\alpha -1  }{10} \cdot n\cdot (1-q)^{5/2}   \mathrm{e}^{-\frac{\pi^2}{6(1-q)}}\cdot (1-q)} \leq  \frac{(1000x)^2}{2\alpha -1}\cdot\frac{ \mathrm{e}^{\frac{\pi^2}{6(1-q)}}\log n}{  n\cdot (1-q)^{7/2}}    \]  The result follows by \eqref{eq:conccuts} as $\Pr{Y_n(\alpha)\neq X_n(\alpha)}\leq n^{-7}$ by Lemma \ref{lem:localify}. \end{proof}

\subsection{Proof of Theorem \ref{prop:sep}} \label{sec:pfof1.3}

We now use our bounds on $\Pr{\mathcal{C}_k^\cf}$ and $X_n(\alpha)$ to establish Theorem \ref{prop:sep}, which shows that $q=1-\frac{\pi^2}{6\log n}$ is a sharp threshold for having a vertex cut separating the graph into two macroscopic pieces.

\begin{proof}[Proof of Theorem \ref{prop:sep}]As we are proving a statement concerning a limit in $n$ we can assume $n\geq\big(\frac{100}{\min\{1-\alpha, \;2\alpha -1\}}\big)^5$. We break the proof into three cases depending on the value of $q$; in the first two cover the $0$-statement, and the last case deals with the $1$-statement.
	
	     Recall that $X_n(\alpha)$ is the number of cut vertices in  $[\lceil (1-\alpha)n\rceil,\lfloor \alpha n\rfloor   ]$  and thus $\{X_n(\alpha)\geq 1 \} =\{\exists \text{ an } (1,\alpha )\text{-separator} \}$. By Markov's inequality, Corollary \ref{cor:sepcor}, and Lemma \ref{lem:R_kbyF_k}, we have  \[ \Pr{X_n(\alpha)\geq 1 } \leq \Ex{X_n(\alpha)} \leq 2 n\cdot \max_{k\in [\lceil (1-\alpha)n\rceil,\lfloor \alpha n\rfloor   ]}\Pr{\mathcal{C}_k^{\mathcal{F}}}\leq 2n\cdot \max_{k\in [\lceil (1-\alpha)n\rceil,\lfloor \alpha n\rfloor   ]}\Pr{\cf_k}.\] To prove the $0$-statement it suffices to show $\Pr{\mathcal{C}_k^{\mathcal{F}}}$ or $\Pr{\cf_k}$ is $\lo{\frac1n}$ when $k\in [\lceil (1-\alpha)n\rceil,\lfloor \alpha n\rfloor ]$.  
	     
	     \medskip
\noindent\textbf{Case (i)} $\left[(1-q)^{-1} \geq n^{4/5}\right]$\textbf{:} If $q=1$ then the result follows from Lemma \ref{treebanddiam1}, so we can assume $q\neq 1$. In this case we have $q\geq 1-n^{-4/5}$ and $\min\{k,n-k \}-1 \geq 5\sqrt{n}$, so 
\begin{equation}\frac{1-q^{\min\{k,n-k \}}}{1-q} = \frac{(1-q)(1 + q +q^2+\cdots +q^{\min\{k,n-k \}-1})}{1-q}  \geq  \sum_{i=0}^{5\sqrt{n}} (1-n^{-4/5})^i. 
\end{equation}Now, by Bernoulli's inequality, we have 
\begin{equation}
\frac{1-q^{\min\{k,n-k \}}}{1-q}\geq \sum_{i=0}^{5\sqrt{n}} (1-i\cdot n^{-4/5}) \geq 4\sqrt{n}. 
\end{equation}
Using this bound in combination with Lemma \ref{lem:cheapflush} and the fact $q\geq 1-n^{-4/5}\geq 1/2$ yields  \[\Pr{\cf_k}\leq \exp\left(- \frac{q }{2}\cdot \frac{ \left(1-q^{\min\{k,n-k \}} \right)}{1-q} \right) \leq \exp\left(-\frac{1}{4}\cdot 4\sqrt{n} \right)\leq \mathrm{e}^{-\sqrt{n}}.\]

		\noindent\textbf{Case (ii)} $[(1-q)^{-1} < n^{4/5},\text{ and } \frac{\pi^2}{6(1-q)} - \log n + \frac{5\log\log n}{2} \rightarrow \infty ]$\textbf{:} By Lemma \ref{lem:precisecor} we have \begin{equation}\label{eq:boundonPS}\Pr{\mathcal{C}_k^{\mathcal{F}}} \leq \mathrm{e}^{6}  \cdot \exp\left(-\frac{\pi^2}{6(1-q)}\right)\cdot (1-q)^{5/2}= \exp\left(6 -\frac{\pi^2}{6(1-q)} + \frac{5\log(1-q)}{2}\right),\end{equation} for  $k\in [\lceil (1-\alpha)n\rceil,\lfloor \alpha n\rfloor ]$ as $n\geq\big(\frac{100}{\min\{1-\alpha, \;2\alpha -1\}}\big)^5$. Differentiating the exponent of \eqref{eq:boundonPS} gives \begin{equation}\label{eq:diff}\frac{\mathrm{d} }{ \mathrm{d} \,q }\left(6 -\frac{\pi^2}{6(1-q)} + \frac{5\log(1-q)}{2}\right) =-\frac{\pi^2}{6(1-q)^2} - \frac{5}{2(1-q)} = \frac{15q -\pi^2 - 15}{6(1-q)^2}. \end{equation} From \eqref{eq:diff} we see that the exponent of the bound on $\Pr{\mathcal{C}_k^{\mathcal{F}}}$ from \eqref{eq:boundonPS} is monotone decreasing in $q$ provided $q\in [0,1)$. Hence to bound $\Pr{\mathcal{C}_k^{\mathcal{F}}}$ it suffices to evaluate \eqref{eq:boundonPS} for the smallest $q$ in the scope of this case. Thus, if we let $w(n)= \frac{\pi^2}{6(1-q)} - \log n + \frac{5\log\log n}{2} $ then it suffices to bound $\Pr{\ce_k}$ in the case $w(n)\rightarrow \infty $ arbitrary slowly and we will assume, for convenience, that $ w(n) \leq \sqrt{\log\log n}$. As $q\geq 1-2/\log n $ we have $\log(1-q) \leq 2 -\log\log n $. Thus by \eqref{eq:boundonPS} \[\Pr{\mathcal{C}_k^{\mathcal{F}}} \leq \exp\left(6 -\frac{\pi^2}{6(1-q)} + \frac{5\log(1-q)}{2}\right) \leq  \exp\left(7 -\log n  - w(n) \right) = \frac{\mathrm{e}^7 \cdot \mathrm{e}^{-w(n)} }{n } = \lo{\frac{1}{n}}. \]  
			
		\noindent	\textbf{Case (iii)} $\left[q \geq 0,\text{ and } \frac{\pi^2}{6(1-q)} - \log n + \frac{9\log\log n}{2} \rightarrow -\infty \right]$\textbf{:} In this final case we will show that $\Pr{X_n(\alpha) \geq 1 }=1-\lo{1} $ by applying Lemma \ref{lem:lowerX_n} which bounds $X_n(\alpha)$ from below.
			If we let $w(n)= - \frac{\pi^2}{6(1-q)} + \log n - \frac{9\log\log n}{2} $ then for this case we have $w(n) \rightarrow \infty$ and 
			\begin{equation}\label{eq:caseivexpbdd}\exp\left( -\frac{\pi^2}{6(1-q)}\right) = \exp\left(- \log n + \frac{9\log\log n}{2} + w(n)   \right) = \frac{\mathrm{e}^{w(n)}\cdot \log^{9/2} n }{n} . \end{equation}
			For this case the (loose) bound $q\leq 1-1/\log n$ holds, thus $ (1-q)^{5/2}\geq 1/ \log^{5/2} n$, and so   \begin{equation}\label{eq:expbdd} n\cdot (1-q)^{5/2}\cdot \mathrm{e}^{-\frac{\pi^2}{6(1-q)}} \geq n \cdot \frac{1}{\log^{5/2} n} \cdot \frac{\mathrm{e}^{w(n)}\cdot \log^{9/2} n }{n} =    \mathrm{e}^{w(n)}\log^2 n ,\end{equation} by \eqref{eq:caseivexpbdd}. The bound $q\leq 1-1/\log n$ also implies $1/(1-q)^{7/2}\leq \log^{7/2} n$ and so, again by \eqref{eq:caseivexpbdd}:
			 \begin{equation}\label{eq:varbdd}\frac{ \log  n }{n(1-q)^{7/2} }\cdot \mathrm{e}^{\frac{\pi^2}{6(1-q)}}  \leq \frac{\log^{9/2} n}{n} \cdot  \frac{n}{\mathrm{e}^{w(n)}\cdot \log^{9/2} n } =\mathrm{e}^{ -w(n)}. \end{equation}

		Inserting the bounds \eqref{eq:expbdd} and \eqref{eq:varbdd} into Lemma \ref{lem:lowerX_n} and choosing $x=w(n)$ gives \[\Pr{X_n(\alpha) < \frac{(2\alpha-1)\log^2n}{10}\cdot \mathrm{e}^{w(n)}\left(1 -\frac{1000}{\sqrt{2\alpha -1}}\cdot  w(n)\cdot \mathrm{e}^{ -w(n)/2}   \right)}\leq \frac{5}{w(n)^2} + \frac{1}{n^7}. \]The result follows since $\alpha >1/2$ is fixed and $w(n)\rightarrow \infty$.\end{proof}

\begin{remark}\label{rmk:sep}The events $\cf_{k} $ are positively correlated, that is $\Pr{\cf_k\cap \cf_{k+j}}>\Pr{\cf_k}\Pr{ \cf_{k+j}}$. This inequality also holds for $\mathcal{F}$ replaced by $\mathcal{L}$ for $j$ not too large. This is an obstruction to proving a (significantly) improved lower threshold in Case $(iii)$ using the second moment method.
\end{remark}

\subsection{Proof of Theorem \ref{diam}}

We now prove Theorem \ref{diam}, which shows the diameter is linear when $q$ is bounded away from $1$. Theorem \ref{diam} follows from the next lemma and an earlier bound on the number of cut vertices.  
\begin{lemma}\label{lem:diambound}Let $G$ be an $n$-vertex graph containing a Hamiltonian path. Then if there exists a set $C\subseteq V$ such that each vertex in $C$ is a cut vertex of $G$, then $\diam(G)\geq |C|+1$. 
\end{lemma}

\begin{proof}Label the vertices from $1$ to $n$ along the Hamiltonian path $P$. Denote the cut vertices by $c_1<  \cdots <   c_{k}$, where $k=|C|$, ordered with respect to the vertex labelling of $P$. \begin{claim} Any path from vertex $1$ to vertex $n$ must include every vertex of $C$.\end{claim}
\begin{poc}
	Since the graph contains an $n$-vertex path as a subgraph (i.e.\ a Hamiltonian path), we must have that for $1\leq i \leq |C|$ the removal of the cut vertex $c_i$ separates the graph into exactly two connected graphs, one containing the subpath $[1,\dots,c_{i}-1] $ and the other containing the subpath $[c_{i}+1, \dots ,n]$, with no edges between them. Now assume, for a contradiction, that there exists a path $Q$ from $1$ to $n$ in $G$ and a vertex $c_i\in C$ which is not contained in $Q$. Since the path does not go through $c_i$  it must use an edge $xy$ where $x\in [1,\dots, c_i-1]$ and  $y\in [ c_i+1, n]$ however this contradicts our earlier observation about the cut vertex $c_i$.  
\end{poc} By the claim, the shortest path connecting $1$ to $n$ has length at least $|C|+1$. This follows as $1$ and $n$ cannot be in $C$, since removing them does not disconnect the graph due to the existence of the path $P$ on $[n]$ as a subgraph. \end{proof}

We note that one cannot completely remove the assumption of a subpath of  length $n$ in Lemma \ref{lem:diambound}. To see this consider the binary tree. We are now ready to prove Theorem \ref{diam}. 

\begin{proof}[Proof of Theorem \ref{diam}] Let $\alpha =2/3$ and recall that $X_n(\alpha)$ denotes the number of cut vertices in $[\lceil (1-\alpha)n\rceil,\lfloor \alpha n\rfloor   ] = [\lceil n/3\rceil,\lfloor 2n/3\rfloor   ]$. Since $0<q\leq 1-\eps$ we have $\mathrm{e}^{\frac{\pi^2}{6(1-q)}} \leq \mathrm{e}^{\frac{\pi^2}{6\eps}} $. Assume that $n\geq 300^5\cdot \mathrm{e}^{\frac{\pi^2}{6\eps}} \geq \left(\frac{100}{\max\{1-\alpha,2\alpha -1  \}} \right)^5 $. We will now apply Lemma \ref{lem:lowerX_n} to bound $X_n(\alpha)$. First observe that, by the restrictions on $n$ above and $0<q\leq 1-\eps$,   \[\frac{200}{\sqrt{2\cdot(2/3)-1}}\sqrt{\frac{\mathrm{e}^{\frac{\pi^2}{6(1-q)}} \log  n }{n(1-q)^{7/2} }}\leq n^{-1/3}, \quad \text{and} \quad  \frac{2\cdot(2/3)-1}{10}\cdot n(1-q)^{5/2} \mathrm{e}^{-\frac{\pi^2}{6(1-q)}} \geq \frac{\eps^{5/2}\mathrm{e}^{-\frac{\pi^2}{6\eps}} }{30} n.\] Thus, is we set $c:=c(\eps) =  \eps^{5/2}\mathrm{e}^{-\frac{\pi^2}{6\eps}} \cdot 300^{-5}  $, then our earlier assumption on $n$ is implied by $n\geq 1/c$. Furthermore, if we set $x=n^{1/6}$ and $\alpha =2/3$ in Lemma \ref{lem:lowerX_n}, then this yields $\Pr{X_n(\alpha)<  c n} \leq  3n^{-1/6}\leq n^{-1/10}$. The result follows as there are at least $X_n(\alpha)$ cut vertices and $\{\diam(\cp(n,q)) \leq  d+1\} \subseteq \{X_n(\alpha) \leq  d \}$ holds for any integer $ d \geq 0 $ by Lemma \ref{lem:diambound}. \end{proof}

\section{The Treewidth and Cutwidth, Proof of Theorem \ref{thm:treepathcut}}\label{sec:treewidth}
The aim of this section is to prove Theorem \ref{thm:treepathcut}, which gives lower and upper bounds on the treewith and cutwidth in the tangled path, respectively. This is an amalgamation of several bounds proved in this section. In particular, the lower bounds on treewidth follow from  Lemmas \ref{trwlower}, \ref{lem:trwlowerhighq}, and the bound $\trw(G) \geq \lfloor \iso(G) \cdot n/4 \rfloor    $ from \eqref{trw-sep} in combination with Theorem \ref{treebanddiam}. Then the upper bounds on cutwidth follow from Lemma \ref{cwupper}, and as $\cw(G) \leq |E(G)|\leq 2n-2$ by \eqref{trw-sep}.

\subsection{Lower Bound on Treewidth}\label{sec:trwlow}
The main results in this section are Lemmas \ref{trwlower} and \ref{lem:trwlowerhighq}, which give two separate bounds on treewidth. The first bound is tight, but it only holds for small values of $q$. Whereas the second holds for any $q\leq 1 - 1/n$, however it is only known to be tight up to a $\log$ factor. Both bounds follow from finding a minor with good expansion properties, however the probability of finding such a minor is calculated slightly differently in the two different cases.

Before we begin in earnest, we must introduce consecutive patterns in permutations. Given a list of $k$ distinct integers $w=w_1,\dots, w_k$, the {\em standardization} of $w$, written $\st(w)$, is the unique permutation of $[k]$ that is order-isomorphic to $w$. That is, we obtain $\st(w)$ by replacing the smallest element among $\{w_1,\ldots,w_k\}$ with $1$, the second smallest with $2$, and so on. We say that $\pi\in S_n$ {\em contains $\sigma\in S_k$ consecutively} if there exists an index $i\in[n-k+1]$ such that $\st\!\left(\pi(i),\dots,\pi(i+k-1)\right)=\sigma$; otherwise, we say that $\pi$ {\em avoids $\sigma$ consecutively}.

We now give a lemma relating consecutive patterns in permutations to minors in the related tangled paths.  Figure \ref{fig:conseqtrw} illustrates this Lemma. 

\begin{lemma}\label{lem:concpattern}
	Let $1\leq k\leq n$ be integers and let $ \pi\in  S_n$ contain $\sigma \in  S_k$ consecutively. Then, $\layer(\pi(P_n),P_n)$ contains $\layer(\sigma(P_k),P_k)$ as a minor. 
\end{lemma}

\begin{proof}Let $\pi(i),\dots,\pi(i+k-1) $ be such that $\st\!\left(\pi(i),\dots,\pi(i+k-1)\right)=\sigma$ and $s_1 < \cdots <   s_{k}$ be the elements of $\{\pi(i),\dots,\pi(i+k-1) \}$ ordered increasingly. Define $P'$ to be the path of length $s_k-s_1+1$ on $ \{s_1, s_1+1,\dots , s_k\}$. Finally, let the graph $H$ be the union of $P'$ and the set of edges $\{\pi(i+j)\pi(i+j+1)\}_{j=0}^{k-2}$. Observe that $H$ is isomorphic to a copy of $\layer(\sigma(P_k),P_k)$ where the edges of $P_k$ have been subdivided. Now, as $H\subseteq \layer(\pi(P_n),P_n)$ the result follows. 
\end{proof}

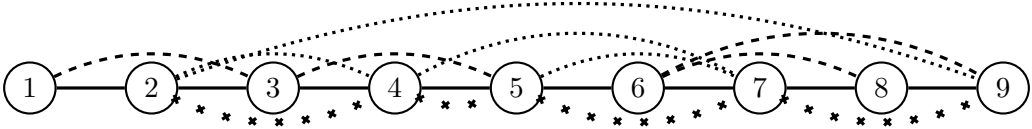
\begin{figure}  
	\center\begin{tikzpicture}[xscale=0.8,yscale=0.6,knoten/.style={thick,circle,draw=black,fill=white},edge/.style={black,very thick},oedge/.style={dotted,very thick}, bedge/.style={dashed,very thick},gedge/.style={decorate,decoration={crosses},very thick}]
	\foreach \x in {1,...,9}
	\node[knoten] (\x) at (2*\x,5) {$\x$};
	\draw[edge] (1) to (2);
	\draw[edge] (2) to (3);
	\draw[edge] (3) to (4);
	\draw[edge] (4) to (5);
	\draw[edge] (5) to (6); 
	\draw[edge] (6) to (7); 
	\draw[edge] (7) to (8);
	\draw[edge] (8) to (9);

	\draw[bedge] (3) to[out=30,in=150 ] (5); 
	\draw[bedge] (1) to[out=30,in=150 ] (3); 
	\draw[bedge] (6) to[out=35,in=145 ] (9);
	\draw[bedge] (6) to[out=30,in=150 ] (8);

	\draw[oedge] (5) to[out=30,in=150 ] (7); 
	\draw[oedge] (4) to[out=35,in=145 ] (7);
	\draw[oedge] (2) to[out=30,in=150 ] (4);
	\draw[oedge] (2) to[out=25,in=155 ] (9);
	
	\draw[gedge] (2) to[out=-30,in=-150 ] (4);
	\draw[gedge] (4) to[out=-25,in=-155 ] (5);
	\draw[gedge] (5) to[out=-30,in=-150 ] (7);
	\draw[gedge] (7) to[out=-30,in=-150 ] (9);

	\end{tikzpicture}

	\caption{\label{fig:conseqtrw}An example of Lemma \ref{lem:concpattern}. Let $\sigma=(3,4,2,1,5)$ and $\pi=(1,3,5,7,4,2,9,6,8)$, thus $\st(5,7,4,2,9) = \sigma$. The graph induced by the crossed and dotted edges is isomorphic to $\layer(\sigma(P_5),P_5)$. The graph induced by the dashed, solid and dotted edges is $\layer(\pi(P_9),P_9)$. The crossed edges can be subdivided to give a path on $\{2,\dots, 9 \}$.}
\end{figure}	
 
  We now recall two results from \cite{BhatSubSeq}, the first is a restricted form of independence for patterns in a $q$-Mallows permutation.

\begin{lemma}[{\cite[Lemma 2.5]{BhatSubSeq}}] \label{lem:indep-seqs} Let $(i_1,\ldots,i_k)$ and $(i'_1,\ldots,i'_\ell)$
	be two increasing sequences such that $i_k < i_1'$. Let $\pi \sim \mu_{n,q}$ for $n\ge i'_\ell$. Then, $\st\left(\pi(i_1),\ldots,\pi(i_k)\right)$ and $\st\left(\pi(i_1'),\ldots,\pi(i_k')\right)$ are independent.
\end{lemma}

The second follows from translation invariance of Mallows permutations \cite[Lemma 2.6]{BhatSubSeq}. 
\begin{lemma}[{\cite[Corollary 2.7]{BhatSubSeq}}]\label{lem:induced_permutation_Mallows}
	Let $(i,i+1,\ldots,i+k-1)\subseteq[n]$ be a sequence of consecutive elements. If $\pi\sim\mu_{n,q}$ then
	$\st\left(\pi(i),\ldots,\pi(i+k-1)\right)\sim\mu_{k,q}$. 
\end{lemma}

Given $\sigma\in  S_k $, we let $\mathsf{Forb}_n(\sigma) \subseteq  S_n$ denote the set of permutations in $ S_n$ that avoid $\sigma$ consecutively. Similarly, given  $S\subseteq  S_k$ we let $\mathsf{Forb}_n(S) = \bigcap_{\sigma\in S}\mathsf{Forb}_n(\sigma)$. Recall the $(n,q)$-Mallows measure $\mu_{n,q}(\sigma) $ of  $\sigma \in  S_n$ from \eqref{eq:mu_n_q_def}. Thus $\mu_{n,q}(S)=\sum_{\sigma\in S}\mu_{n,q}(\sigma) $ for any $S\subseteq  S_n$.

\begin{lemma}\label{lem:patterncouplesimple}Let $1\leq k\leq n$ be integers, and $0\leq q\leq 1$. Then, for any $S\subseteq  S_k$ we have 
	\[\mu_{n,q}\left(\mathsf{Forb}_n(S)\right) \leq \mu_{k,q}\left(\mathsf{Forb}_k(S)\right)^{\lfloor n/k \rfloor }.  \]     
\end{lemma}
\begin{proof}To begin, observe that if a permutation $\pi\in  S_n$ avoids a set of permutations $S\subseteq  S_k$ consecutively then $\st\!\left(\pi(ik+1),\dots,\pi(ik+k)\right)\notin S $ for any $0\leq i\leq \lfloor n/k \rfloor -1$. The permutations $\left(\st\!\left(\pi(ik+1),\dots,\pi(ik+k)\right)\right)_{i=0}^{\lfloor n/k \rfloor -1}$ are independent by Lemma \ref{lem:indep-seqs} and by Lemma \ref{lem:induced_permutation_Mallows} each permutation $\st\!\left(\pi(ik+1),\dots,\pi(ik+k)\right)$ has distribution $\mu_{k,q}$. Thus, for any $S\subseteq  S_k$, we have
	\begin{equation*} \mu_{n,q}\left(\mathsf{Forb}_n(S)\right) \leq \prod_{i=0}^{\lfloor n/k \rfloor -1}\Pr{\st\!\left(\pi(ik+1),\dots,\pi(ik+k)\right)\notin S}= \mu_{k,q}\left(\mathsf{Forb}_k(S)\right)^{\lfloor n/k \rfloor}, \end{equation*}as claimed. 
\end{proof}

We now have what we need to prove our fist lower bound on treewidth.

\begin{lemma}\label{trwlower}Let $n\geq 1$ be an integer, $\kappa>0$ be any constant, and $0< q \leq  1 - \kappa(\log\log n)^2/\log n$. Then, there exists a constant $c>0$ such that for any $n> 1/c$ we have  	\[\Pr{\trw(\cp(n,q)) <c\left(\sqrt{\frac{\log n}{\log (1/q)} } + 1\right)} \leq  \exp(-\sqrt{n} ).\]
\end{lemma}

\begin{proof} By \eqref{eq:cwdef} and Lemma \ref{treebanddiam1}, for any $k$, the bound $\trw(\cp(k,1))\geq k/200$ holds with probability $1-f(k)$ where $f(k)\rightarrow 0$ as $k\rightarrow \infty$. Thus, there exist fixed constants $C>0$ such that for any $k\geq C$ there is at least one permutation $\sigma_k\in S_k$ satisfying $\trw\left(\layer\left(\sigma_k(P_k), P_k\right)\right) \geq  k/200$.  We wish to show that, for some $k:=k(n,q)$ defined later, a permutation $\pi\sim \mu_{n,q}$ contains a given $\sigma_k\sim \mu_{k,q}$ as a consecutive pattern. We will do this by appealing to Lemmas \ref{lem:indep-seqs} and \ref{lem:induced_permutation_Mallows}. The result will then follow from Lemmas \ref{lem:patterncouplesimple}  and \ref{sub-div}. To begin, by \eqref{eq:mu_n_q_def} we have 
	\begin{equation}\label{eq:lbddonE_j}\mu_{k,q}(\sigma_k) = q^{\inv(\sigma_k)} \cdot (Z_{k,q})^{-1} = q^{\binom{k}{2}}\cdot \prod_{i=1}^{k} \frac{(1 - q)}{1-q^{i}} \geq   q^{k^2}\cdot (1-q)^k .\end{equation}  Recall that $0< q\leq 1 - \kappa(\log\log n)^2/\log n$, for some $\kappa>0$ and let\begin{equation}\label{eq:alphatree} k =  \sqrt{\frac{\alpha \log n}{\log(1/q)  }} \qquad \text{where} \qquad \alpha=\min\left\{\frac{1}{100},\; \frac{\kappa}{25} \right\}. \end{equation}Now, since $\log q \leq q-1 $ by \eqref{logbdd} we have $1/\log(1/q)\leq 1/(1-q)\leq \log(n) /(\kappa (\log\log n)^2)$. Thus,\begin{equation}\label{eq:bddonk}
		k \leq\sqrt{\frac{\alpha \log^2 n}{\kappa (\log\log n)^2}} =  \sqrt{ \frac{\alpha }{\kappa} }\cdot  \frac{  \log  n}{  \log\log  n}\leq \frac{  \log  n}{ 5  \log\log  n}. 
	\end{equation}Returning to the bound on $\mu_{k,q}(\sigma_k) $, by \eqref{eq:lbddonE_j} and \eqref{eq:bddonk}, for any  large enough $n$ we have
	\begin{equation}\label{eq:interprob} \mu_{k,q}(\sigma_k)\geq q^{k^2}\cdot (1-q)^{k} \geq \mathrm{e}^{ k^2 \log q}\cdot \left(\frac{\kappa (\log\log n)^2}{\log n}\right)^{k}    \geq  \mathrm{e}^{-\alpha\log n}\cdot \left(\mathrm{e}^{- \log\log n}  \right)^{\sqrt{ \alpha /\kappa }\cdot  \frac{  \log  n}{  \log\log  n}}   . \end{equation}
	Now by our choice of $\alpha$ in \eqref{eq:alphatree} we have $\alpha+  \sqrt{\alpha /\kappa }   \leq  1/100 + \sqrt{\kappa/(25\kappa)} < 1/4$ . Thus by \eqref{eq:interprob}: 
	\begin{equation}\label{totprob'} \mu_{k,q}\left(\mathsf{Forb}_k(\sigma_k)\right)= 1- \mu_{k,q}(\sigma_k) \leq 1-  n ^{- \left(\alpha  +\sqrt{\alpha /\kappa } \right)   } \leq 1- n^{-1/4}.   \end{equation} Observe that, for suitably large $n$,  $\lfloor n/k \rfloor \geq n^{3/4}$  by \eqref{eq:bddonk}. Thus, by Lemma \ref{lem:patterncouplesimple} and \eqref{totprob'},   
		\begin{equation}\label{eq:bddonmusmall} \mu_{n,q}\left(\mathsf{Forb}_n(\sigma_k)\right) \leq \mu_{k,q}\left(\mathsf{Forb}_k(\sigma_k)\right)^{\lfloor n/k \rfloor } \leq  \left(1- n^{-1/4} \right)^{n^{3/4}}   \leq \mathrm{e}^{-\sqrt{n}}  .  \end{equation}
	It follows from \eqref{eq:alphatree} that there is a constant $C'$ such that if $n\geq C'$, then $k\geq C$. Thus if we take $c\leq \min\{\sqrt{\alpha}/200,\, 1/C'\}$ so that $n\geq 1/c$ is sufficiently large, then \[
		\Pr{\trw\left(\cp(n,q)\right)<c\cdot \sqrt{\frac{\log n}{\log(1/q)  }}} \leq \mu_{n,q}(\mathsf{Forb}_n(\sigma_k)) \leq \mathrm{e}^{-\sqrt{n}},  
	\] by  Lemmas \ref{lem:patterncouplesimple} and \ref{sub-div}, \eqref{eq:bddonmusmall}, and the expression for $k$ given by \eqref{eq:kexpression}. 
\end{proof}

The proof of the following lower bound is similar in its use of Lemma \ref{lem:concpattern}, but to calculate the probability it uses Lemma \ref{lem:largeqpatterncouple} to relate the $q$-Mallows measure to the $1$-Mallows measure.  

\begin{lemma}\label{lem:trwlowerhighq}Let $n\geq 1$ be an integer, and $q\leq 1 - \frac{1}{n}$ satisfy $\lim_{n \rightarrow \infty} q =1$. Then, for sufficiently large $n$,  \[\Pr{ \trw\left(\cp(n,q) \right) <   (1-q)^{-1} /10^5}\leq \exp\left( - n/100 \right) .\]
\end{lemma}

\begin{proof} 
	For an integer $k\geq 1$, we define the set  
	\[ T_k=\left\{\sigma \in  S_k : \trw(\layer(\sigma(P_k),P_k))\geq   k/50 \right\}\subseteq  S_k.\]  By Lemma \ref{lem:expander} $\Pr{\iso(\cp(n,1))\leq \frac{1}{40}}\leq 1000  \cdot k^{7/2}\cdot    \left( \frac34\right)^{k}$ holds for any $k\geq 100$. By \eqref{trw-sep} for any $k$-vertex graph $G$ we have $\trw(G) \geq \lfloor \iso(G)\cdot k \rfloor  -1 $. Thus we conclude that for $k\geq 100$,
	\begin{equation}\label{eq:kavoidbdd}\mu_{k,1}(\mathsf{Forb}_k(T_k)) \leq 1000\cdot   k^{7/2}\cdot    \left( \frac34\right)^{k} \leq \mathrm{e}^{-k/20}.\end{equation} 
	We now fix	
	\begin{equation}\label{eq:kexpression}k = \left\lfloor \frac{1}{1000}\cdot   \frac{1}{1-q} \right\rfloor. \end{equation} Since $\lim_{n \rightarrow \infty} q =1$, we can assume that $k$ satisfies $k\geq 100$ by taking a suitably large $n$. By Lemma \ref{lem:largeqpatterncouple}, \eqref{eq:kavoidbdd} and \eqref{eq:kexpression}, we have  
	\[\mu_{k,q}(\mathsf{Forb}_k(T_k))\leq \mathrm{e}^{9k^2(1-q)}\cdot \mu_{k,q}(\mathsf{Forb}_k(T_k)) \leq \mathrm{e}^{\frac{9}{1000}\cdot k }\cdot \mathrm{e}^{-k/20} \leq \mathrm{e}^{-4k/100}.  \]
Note that $k\leq \lfloor n/10\rfloor$ by \eqref{eq:kexpression} and $q\leq 1-\frac{1}{ n}$. Thus, by Lemma \ref{lem:patterncouplesimple}, for  sufficiently  large $n$,  
	\begin{equation}\label{eq:bddonmu} \mu_{n,q}\left(\mathsf{Forb}_n(T_k)\right) \leq \mu_{k,q}\left(\mathsf{Forb}_k(T_k)\right)^{\lfloor n/k \rfloor } \leq \mathrm{e}^{-(4k/100)\cdot \lfloor n/k \rfloor} \leq \mathrm{e}^{-n/100}  .  \end{equation}

	Recall that $\mu_{n,q}(\mathsf{Forb}_k(T_k))$ is the probability $\pi \sim \mu_{n,q}$ avoids all permutations $\sigma  \in T_k$ consecutively. Thus, for any $\pi \in  S_n$ such that $\pi\notin \mathsf{Forb}_n(T_k)$, Lemmas \ref{sub-div} and \ref{lem:concpattern} yield   \[\trw\left(\layer(\pi(P_n),P_n)\right) \geq \min_{\sigma\in T_k}\trw\left(\layer(\sigma(P_k),P_k)\right)\geq  k/50.\] Recall also that $ \layer(\pi(P_n),P_n) \sim \cp(n,q)$ as $\pi \sim \mu_{n,q}$. Thus by \eqref{eq:bddonmu}, for large $n$, we have \begin{equation*} 
		\Pr{\trw\left(\cp(n,q)\right)<k/50} \leq\mu_{n,q}(\mathsf{Forb}_n(T_k)) \leq \mathrm{e}^{-n/100}. 
	\end{equation*} The result now follows from the expression for $k$ given by \eqref{eq:kexpression}. \end{proof}

\subsection{Upper Bound on Cutwidth} \label{sec:cwupper}
Recall the definition of the cutwidth of a graph $G$ given by \eqref{eq:cwdef}: 
\begin{equation}\label{eq:cut}\cw(G)= \min_{f:V\rightarrow \mathbb{Z}, \text{ injective}} \; \max_{x\in \mathbb{R}}\;\left| \left\{ ij \in E(G) : f(i)\leq x <f(j)\right\}\right|.\end{equation}In this section we  prove Lemma \ref{cwupper}, which gives the upper bounds on cutwidth in Theorem \ref{thm:treepathcut}. 

\begin{lemma}\label{cwupper}Let $n\geq 1$ be an integer, $c>0 $ be any constant, and $0\leq q \leq  1 -  c(\log\log n)^2 / \log n$. Then there exists a constant $C $ such that \[\Pr{\cw(\cp(n,q)) >C\left(\sqrt{\frac{\log n}{\log (1/q)} } + 1\right)} =\BO{\frac{1}{n^3}} .\]
	Additionally, if $q\geq 1 - (\log\log n)^2 /\log n$, then \[\Pr{\cw(\cp(n,q)) > \frac{5}{1-q}\log\left(\frac{1}{1-q} \right) } =\BO{\frac{1}{n^3}}.\]
\end{lemma}

We will begin by outlining a sketch of the proof of Lemma \ref{cwupper}. 

\textit{Proof Sketch:} For an upper bound on $\cw\left(\cp(n,q)\right)$ we fix $f:[n] \rightarrow  [n]$ in \eqref{eq:cut} to be the identity map denoted $\ID$. That is, we order the vertices of $\cp(n,q)$ using the ordering of the un-permuted path $P_n$ with edges $(i)(i+1)$ for $i\in [n-1]$. We bound the number of edges crossing any vertex $i$ by bounding the number of elements with values $j>i$ are inserted next to elements $k<i$ by the $q$-Mallows process. To do this we show that, for $b_n= \lceil \frac{8\log n}{\log(1/q)}\rceil$ given by \eqref{eq:b} and some suitable $\ell,L$ where $L\geq \ell$,  the following events hold with high probability: \begin{enumerate}[label=(\roman*)]\itemsep=0pt
\item\label{ustep1} no insert position $v_i$ has value greater than $b_n$,
    \item\label{ustep3} after $L$ steps the leftmost $b_n$ places contain only elements added at most $L$ steps ago, 
        \item\label{ustep2} within any window of $L$ steps there are at most $\ell$ values of $v_i$ greater than $\ell$.
\end{enumerate}

The events \ref{ustep1} and \ref{ustep2} control the number of long edges created from new entries being added far away from the left-hand end of the process. The event \ref{ustep3} is  more subtle, it ensures that the left-hand end of the permutation grown by the $q$-Mallows process cannot retain entries that were inserted long ago, again preventing long edges caused by new elements lying next to old ones. We show that if \ref{ustep1} -  \ref{ustep2} hold, then the number of edges crossing any vertex is $\mathcal{O}(\ell)$.

Having concluded the proof sketch we now formalise event described in \ref{ustep3}. To do this we introduce the sparse flush event, which is a relaxation of the local flush event $\mathcal{L}_k$ given by \eqref{eq:localflushingevent}. For positive integers $n,k,b$ where  $1\leq k\leq n$ and $b\leq n-k$ and real $L\leq n-k$ we say there is a \textit{sparse flush} $\mathcal{S}(k,b,L)$ of $b$ items at step $k$ with length $L$ if the following event holds:
\begin{equation}\label{eq:sparseflushingevent}
\mathcal{S}(k,b,L)=\left\{\text{there exist $k\leq t_1< \cdots < t_{b} \leq k+L $ such that $v_{t_i}\leq i $ for all $1\leq i \leq b$}\right\}.
\end{equation}

To give an intuition on the sparse flush we must first recall the local flush event \eqref{eq:localflushingevent} given by $\mathcal{L}_k= \left\{\text{for each } k <i \leq  k+ b_n(q)\text{ we have }v_i\leq i-k \right\} $, where $b_n=\left\lceil 8\log(n)/ \log(1/q) \right\rceil$. Recall also that no insert position is greater than $b_n$ w.h.p.\ by \eqref{eq:probnobigv}. Thus, if the local flush event $\mathcal{L}_k$ holds for some vertex $k$ then w.h.p., no element $j\geq k +b_n$ is inserted next to any element $i\leq k$. If we choose $b_n=\left\lceil 8\log(n)/ \log(1/q) \right\rceil$ in \eqref{eq:sparseflushingevent} then, conditional on $\mathcal{S}(k,b_n,L)$, the values $(v_{t_i})_{i=1}^{b_n}$ form a (non-consecutive) local-flush, that is $v_{t_1}\leq 1, v_{t_2}\leq 2,\dots, v_{t_{b_n}}\leq b_n$. This sequence of insert positions ensures that w.h.p.\ no element $j\geq k +L$ is inserted next to any element $i\leq k$, and thus no such edge $ij$ is present in the permuted path. If $q$ is very close to $1$, the local flush event only occurs with small probability; the next two lemmas show that if we take $L$ significantly larger than $b_n$, then the sparse flush holds w.h.p., even for $q$ tending to $1$ quite fast.

\begin{lemma}\label{lem:sparseflushlem}Let $n,b\geq 1$ be integers, $0<q< 1$, and $\lambda\geq 1$. Let $L= \lambda \left(b + \frac{q}{1-q} + \frac{1}{\log q}\log \frac{1-q}{1-q^{b}}  \right)$ and $1\leq k\leq n-L$ be an integer. Then, \[\Pr{ \mathcal{S}(k,b,L) ^c}\leq  \left(\frac{(1-q)q^b}{1-q^{b}} \right)^{\lambda -1- \log \lambda}.\]
\end{lemma}

\begin{proof}Recall that each element $x$ in the $q$-Mallows process $(r_n)$ is inserted at step $x$ and relative position $v_x\sim\nu_{x,q}$. The sparse flush $\mathcal{S}(k,b,L)$ consists of a sequence of $b$ steps $(t_{i})_{i=1}^b$ where element $t_i$ is inserted at relative position at most $i$. We call these steps \textit{good}. Conditional on $t_{i-1}$, the $(i-1)$\textsuperscript{th} good step, we define the random variable $\tau_i = t_i - t_{i-1}$ to be the additional number of steps we must wait for the $i$\textsuperscript{th} good step (where $1\leq i\leq b$ and $t_0=k$). 
	
	Let $T = \sum_{i=1}^b \tau_i$ be the total number of steps we have to wait to have $b$ good steps and observe that $\mathcal{S}(k,b,L) =\{T\leq L  \}$. We now aim to bound $T$. Since $(v_x)_{x\geq 0}$ are independent and $\Pr{v_k\leq i} \geq 1- q^i$ by \eqref{eq:Probv_klowerthani}, it follows that $\tau_i\preceq X_i$, for independent $X_i \sim \geo{p_i}$, where $p_i = 1-q^i$. That is, the time between the $i-1$\textsuperscript{th} and $i$\textsuperscript{th} good step is stochastically dominated by a geometric random variable with success probability $1-q^i$. We now set
	\[X = \sum_{i=1}^b X_i \quad \text{ and }\quad \mu= \Ex{X} = \sum_{i=1}^b \frac{1}{p_i}.\]Observe that $T\preceq X$ and so $\Pr{\mathcal{S}(k,b,L)^c}  \leq \Pr{X>L}$. We bound the latter probability using Lemma \ref{lem:jansontail}. To do this we need bounds on $\mu = \Ex{X}$. For the upper bound: 
	\begin{equation*}\mu = \sum_{i=1}^b \frac{1}{1 - q^i} \leq \frac{1}{1-q} + \int_{1}^b\frac{1}{1-q^x}\;\mathrm{d} x= \frac{1}{1-q} + \int_{1}^b\left(1 + \frac{q^x}{1-q^x}\right)\;\mathrm{d} x,  \end{equation*}since $1/(1-q^x) = 1 +q^x/(1-q^x)$ is decreasing in $x$. Thus, by integrating we have
	\begin{equation}\label{eq:sumbdd1}
	\mu\leq  \frac{1}{1-q} + b -1 -  \left[\frac{\log\left(1-q^x\right)}{\log q}   \right]_{x=1}^b = b + \frac{q}{1-q} + \log\left(\frac{1-q}{1-q^{b}} \right)/\log(q).  \end{equation} Similarly, we can obtain the following lower bound:
	\begin{equation}\label{eq:sumbdd2}\mu= \sum_{i=1}^b \frac{1}{1 - q^i}\geq \int_{1}^{b}\frac{1}{1-q^x}\;\mathrm{d} x \geq b + \log\left(\frac{1-q}{1-q^{b}} \right)/\log(q).\end{equation}   
	We apply Lemma \ref{lem:jansontail} to $X$ where $p_*=\min_{i}p_i\geq 1-q$, $\mu=\Ex{X}$ and $\lambda>1$: 
	\begin{align*} 
	\Pr{ \mathcal{S}(k,b,\lambda\mu) ^c} \leq \Pr{X \geq \lambda\mu } \leq \lambda^{-1}(1-p_*)^{\mu \left(\lambda -1- \log \lambda\right)}\leq q^{\mu \left(\lambda -1- \log \lambda\right).}
	\end{align*}We now insert the lower bound on $\mu$ from \eqref{eq:sumbdd2} into the bound in the line above to give
	\begin{align*}\Pr{ \mathcal{S}(k,b,\lambda\mu) ^c}   \leq  q^{\left(b + \log\left(\frac{1-q}{1-q^{b}} \right)/\log(q)\right)\cdot  \left(\lambda -1- \log \lambda\right)}  = \left(\frac{(1-q)q^b}{1-q^{b}} \right)^{ \lambda -1- \log \lambda}. \end{align*}To conclude we observe that $L \geq \lambda \mu$ by \eqref{eq:sumbdd1}. \end{proof} 
We now plug some specific values into Lemma \ref{lem:sparseflushlem} for use in the proof of Lemma \ref{cwupper}.

\begin{lemma}\label{cor:probS} Let $n\geq 1$ be an integer, $0<q < 1$, and $b_n= \left\lceil \frac{8\log n}{\log(1/q)}\right\rceil\geq 1$. Then for any  $L\geq   \frac{100}{1-q}\left(\frac{1}{1-q} +\log n \right)$ and integer $1\leq k \leq n- L$ we have $\Pr{ \mathcal{S}(k,b_n,L) ^c} = \lo{n^{-10}}.$
\end{lemma}
\begin{proof}Observe that $q^{b_n}\leq n^{-8} $ and $\mathcal{S}(k,b_n,L) \supseteq \mathcal{S}(k,b_n,D)$ for any $L\geq D $. Now let
	\begin{equation*}D=10 \left(b_n + \frac{q}{1-q} + \log\left(\frac{1-q}{1-q^{b_n}} \right)/\log(q) \right). \end{equation*} Applying the inequalities  $-1/\log q\leq 1/(1-q) $ and $-\log(1-q)\leq 1/(1-q)$ yields \[D\leq 10 \left(-\frac{8\log n}{\log q } + 1 + \frac{q}{1-q} - \frac{\log\left(1-q  \right)}{-\log(q)} \right)  \leq \frac{100}{1-q}\left(\frac{1}{1-q} +\log n \right).\]Finally, $\Pr{ \mathcal{S}(k,b_n,D) ^c}\leq  \left(\frac{(1-q)q^b_n}{1-q^{b_n+1}} \right)^{10 -1- \log 10} \leq \left(\frac{n^{-8} }{1-n^{-8}} \right)^{6}=\lo{n^{-10}}$ by Lemma \ref{lem:sparseflushlem}.
\end{proof}

We are now ready to prove Lemma \ref{cwupper}, which provides upper bounds on the cutwidth.

\begin{proof}[Proof of Lemma \ref{cwupper}] As mentioned at the start of this section we bound the cutwidth from above by fixing the injection $f:[n] \rightarrow  [n]$ in  \eqref{eq:cut} to be the identity $\ID:[n]\rightarrow [n]$.

	Before beginning in earnest, we treat the case $0\leq q\leq 1/n^2$ as this is a straightforward deduction from an earlier result. Indeed, if $q< 1/n^2$ then by Theorem \ref{thm:displacement} we have 	\begin{equation}\label{eq:smallq}
	\Pr{\cw(\cp)\geq 2}\leq \sum_{i\in[n]} \Pr{|\sigma(i)-i|\ge 2}\le n\cdot 2q^2 = \BO{n^{-3}}.
	\end{equation} 
	
	Thus, from now on we can assume $q\geq 1/n^2$. The remainder of the proof is as follows. First, we introduce the event $\mathcal{W}(\ell,L)$, which essentially controls how many long edges originate from any small set of consecutive vertices. We then prove a claim bounding the cutwidth conditional on this event. To apply the claim we then bound the probability that the event fails, breaking into two cases for different values of $q$. Finally, we conclude by relating the different cases to the lemma's statement. We now begin with the formal definition of the event $\mathcal{W}(\ell,L)$.

	Let $b_n=\left\lceil  8\log(n) / \log(1/q)\right\rceil$ and define the event $\ce=\bigcap_{k\in[n]}\{v_k\leq b_n \} $. For integers $i\in[n]$  and $L\geq \ell \geq 1$, we recall the sparse flush event $\mathcal{S}(i,b_n,L)$ from \eqref{eq:sparseflushingevent}, and also define the event
	\begin{align*} 
	\mathcal{B}(i,\ell,L) &=\left\{\text{there are at most $\ell$ values $k\in [i,\max\{ i+L, n\}]$ such that $\nu_k >\ell$}\right\}. 
	\end{align*}
	
	The aforementioned event $\mathcal{W}( \ell,L)$ is an intersection of these previous three events
	\[\mathcal{W}( \ell,L)= \ce\cap\left(\bigcap_{1\leq i\leq n-L}\mathcal{S}(i,b_n,L)\right)\cap \left(\bigcap_{i\in [n]}\mathcal{B}(i,\ell,L)\right). \] 
	The first two intersected events in $\mathcal{W}( \ell,L)$ imply that for every element $i$ inserted by the $q$-Mallows process, at most the $L$ elements following it can potentially create an edge that crosses the vertex $i\in[n]$. The final event in the intersection states that of these $L$ elements, all but at most $\ell$ will be inserted in the $\ell$ closest positions to the left-hand end of the permutation at the time of insertion. The next claim shows that if this event holds then the cutwidth is bounded.

	\begin{claim}
		For any integers $L\geq \ell\geq 1 $ we have $\mathcal{W}( \ell,L)\subseteq  \{\cw(\cp) \leq 4(\ell +1) \}. $
	\end{claim}
	\begin{poc} Recall that we fixed the injection $f$ in the definition of the cutwidth \eqref{eq:cut}  to be $\ID$. Under $f$ the edges of the un-permuted path $P_n\subseteq \layer(\sigma(P_n),P_n)$ contribute at most $1$ to the cutwidth. We now bound the contribution by the edges of the permuted path $ \sigma(P_n)\subseteq \cp(n,q)$.    
		
		Given vertices $i,j,k\in [n]$ we say that the edge $jk\in E(\sigma(P_n))$ is \textit{bad} (with respect to  $i$) if $j<i$ and $k> i$. We claim that if $\mathcal{W}( \ell,L)$ holds then, for any $i\in [n]$, there are at most $4\ell $ edges which are bad for $i$. We show this by keeping track of where elements are inserted during the $q$-Mallows process $(r_a)_{a=1}^n$. By definition if an edge is bad for $i$, then exactly one of its endpoints must have value greater than $i$. If $i=n$ then there are no bad edges for $i$, so we assume that $i<n$. Observe for any element $k\in\{i+1,\dots, n\}$ there are at most $2$ bad edges with endpoint $k$, since $k$ can be adjacent to at most two elements $j<i$ in the final permutation $r_n$. We now partition the elements $k\in \{i+1,\dots, n\}$ into three  disjoint sets $A_{i}$, $B_i$ and $C_i$  based on their value $k$ and $v_k$, the insert position of the element $k$ relative to the left-hand end of $r_{k-1}$:
		\[A_i =\{k\in \{i+1, \dots, i+L \} : v_k> \ell\},\quad B_i =  \{i+1, \dots, i+L \} \backslash A_i,  \quad\text{and} \quad C_i = \{i+L+1, \dots, n \} .\] See Figure \ref{fig:cutwidth} for an example. We will count the contribution of each set to the bad edges. 	
		
		\textit{Contribution from $(A_i)$}: Conditional on $\mathcal{B}(i,\ell,L) $, we have $|A_i|\leq \ell$. Thus, in total the elements in $A_i$ contribute at most $2\ell$ bad edges. 
		
		\textit{Contribution from $(B_i)$}: The left-most $\ell$ positions of $r_i$ each have value at most $i$.  By the definition of the $q$-Mallows process \eqref{eq:MallowsProc-1}, no additional element with value at most $i$ can ever occupy any of the $\ell$ left-most positions. Since each element $k\in B_i$ is inserted within the $\ell$ left-most  places, there can only be $2\ell$ total bad edges for $i$ with an endpoint $k \in B_i$. 
		
		\textit{Contribution from $(C_i)$}: Observe that if $i\geq n-L$ then $C_i$ is empty. Otherwise, as $\mathcal{S}(i,b_n,L)$ holds, we have $r_{k}(j)>i$ for all $k>i+L$ and $j\leq  b_n$. That is, after step $i+L$ of the $q$-Mallows process the permutation will only have elements with value at least $i$ in its $b_n$ left-most positions. However, conditional on $\ce $, no element is inserted at position greater than $b_n$ relative to the left-hand end of the permutation. Thus, the elements $k>i+L$ create no bad edges.
		
		Collecting these contributions, if we condition on $\mathcal{W}(\ell,L)$ then there are in total at most $4\ell$ bad edges for any $i\in [n]$. For any $x\in \mathbb{R}$ there exists an $i$, such that all edges crossing $x$ are either bad for $i$ or have $i$ as an endpoint. The result follows as there are at most $4$ edges with endpoint $i$, since $\cp(n,q)$ has maximum degree $4$.  
	\end{poc}

	\begin{figure}  
		\center\begin{tikzpicture}[xscale=0.5,yscale=0.6,knoten/.style={thick,minimum size=.62cm,circle,draw=black,fill=white},gnoten/.style={thick,minimum size=.62cm,draw=black},onoten/.style={ thick,minimum size=.5cm ,shape=diamond,draw=black},pnoten/.style={ thick,minimum size=.62cm,regular polygon,regular polygon sides=6,draw=black},edge/.style={black,very thick}, bedge/.style={color=black,very thick},dotedge/.style={dotted,very thick}]
		\foreach \x in {1,...,3}
		\node[knoten] (\x) at (2*\x,5) {$\! \x \!$};
		\foreach \x in {4,9}
		\node[gnoten] (\x) at (2*\x,5) {$\! \x \!$};
		\foreach \x in {5,6,7,8}
		\node[onoten] (\x) at (2*\x,5) {$\! \x \!$};
		\node[onoten] (10) at (20,5) {$\!\! 10 \!\!$};
		\node[onoten] (11) at (22,5) {$\!\! 11 \!\!$};
		\foreach \x in {12,...,15}
		\node[pnoten] (\x) at (2*\x,5) {$\!\!\!\x\!\!\!$};
		
		\draw[edge] (1) to (2);
		\draw[edge] (2) to (3);
		\draw[edge] (3) to (4);
		\draw[edge] (4) to (5);
		\draw[edge] (5) to (6); 
		\draw[edge] (6) to (7); 
		\draw[edge] (7) to (8);
		\draw[edge] (8) to (9);
		\draw[edge] (9) to (10);
		\draw[edge] (10) to (11);
		\draw[edge] (11) to (12);
		\draw[edge] (12) to (13);
		\draw[edge] (13) to (14);
		\draw[edge] (14) to (15);

		\draw[bedge] (13) to[out=35,in=145 ] (15);
		\draw[bedge] (13) to[out=25,in=155 ] (14);
		\draw[bedge] (10) to[out=28,in=152 ] (14);
		\draw[bedge] (10) to[out=28,in=152 ] (11);
		\draw[bedge] (11) to[out=25,in=155 ] (12);
		\draw[bedge] (8) to[out=25,in=155 ] (12);
		\draw[bedge] (6) to[out=32,in=148 ] (8);
		\draw[bedge] (2) to[out=25,in=155 ] (6);
		\draw[bedge] (2) to[out=25,in=155 ] (7);
		\draw[bedge] (7) to[out=32,in=148 ] (9);
		\draw[bedge] (5) to[out=30,in=150 ] (9);
		\draw[bedge] (3) to[out=30,in=150 ] (5);
		\draw[bedge] (1) to[out=25,in=155 ] (3);
		\draw[bedge] (1) to[out=25,in=155 ] (4);

		\draw[dotedge] (7,5) to (7,6.5);
		\draw[dotedge] (7,5) to (7,4.4);
		
		\end{tikzpicture}
		\vspace*{-2mm}
		
		\caption{\label{fig:cutwidth}A visual aid to the claim in proof of Lemma \ref{cwupper}. Observe the `random' sequence $\mathbf{x}=(1,1,2,4,2,1,3,1,5,1,2,3,1,2,1)$ satisfies $\mathcal{W}(\ell,L)$ where with $\ell =3 $, $L=8$ and $b=5$. The permutation generated by the $q$-Mallows process with input $\mathbf{x}$ is $\pi=(15,13,14,10,11,12,8,6,2,7,9,5,3,1,4)$. We let $x=3.5$ be the location of the cut (as this maximises the cut) and observe that there are $3$ bad edges for vertex $3$. We also have the sets $A_3=\{4,9\}$, $B_{3} =\{5,6,7,8,10,11\}$ and $C_3=\{12,13,14,15\}$ whose vertices are represented as squares, diamonds and hexagons respectively.}
	\end{figure}
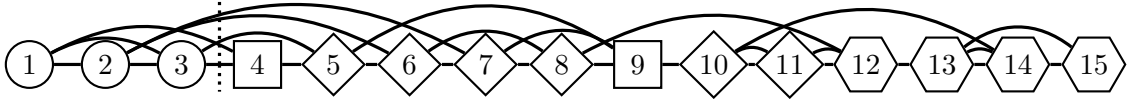

	 We now bound the failure probability of each of the constituent events in $\mathcal{W}(\ell,L)$. To begin, we set 
	\begin{equation*} L= \frac{100}{1-q}\left(\frac{1}{1-q} +\log n \right).\end{equation*}

	Recall that $q\geq 1/n^2$ and so $b_n=\left\lceil  8\log(n)/\log(1/q)\right\rceil \geq \left\lceil  8\log(n)/(2\log n)\right\rceil \geq 1$. Thus, by \eqref{eq:probnobigv}, and Lemma \ref{cor:probS} (since $b_n\geq 1$) plus the union bound, respectively, we have 
	
	\begin{equation}\label{eq:probES} \Pr{\ce^c}\leq n^{-7}  \qquad \text{and} \qquad 
	\Pr{\cup_{1\leq i\leq n-L}\left( \mathcal{S}_{i,b_n,L}\right)^c} \leq n\cdot \lo{n^{-10}}=\lo{n^{-9}}.
	\end{equation}
	
	We bound $\Pr{\mathcal{B}(i,\ell,L)^c}$ by showing the number of large inputs to the $q$-Mallows process is stochastically dominated by a sum of i.i.d.\ Bernoulli random variables. Observe that $\Pr{v_k > \ell}\leq q^\ell$ for any integers $\ell\leq k $ by \eqref{eq:nuupperbdd}, and $\Pr{v_k > \ell} =0$ if $k<\ell$. Let $\{X_j\}_{j \in [L]}$ be a set of independent Bernoulli random variables with success probability $q^\ell$. Then,  since each $v_k$ is independent, for any integers $\ell,L\geq 1$ and $i\in [n]$: \begin{equation*} \left|\{ k\in [i, i+L]: v_k > \ell\}\right| \preceq \sum_{j=1}^L X_j.\end{equation*}
	If $\ell>Lq^\ell$ and $\ell >\mathrm{e}$, then taking $\delta = \ell/(Lq^\ell) -1 >0$ in Lemma \ref{Chertail} we have
	\begin{equation}\label{eq:tailbdd}\Pr{ \sum_{j=1}^L X_j>\ell } \leq \left(\frac{\mathrm{e}^\delta }{(1+\delta)^{1+\delta} } \right)^{Lq^\ell}\leq \left(\frac{\mathrm{e}^{\frac{\ell}{Lq^\ell}-1} }{(\frac{\ell}{Lq^\ell})^{\frac{\ell}{Lq^\ell}} } \right)^{Lq^\ell} \leq \left(\frac{\mathrm{e}Lq^\ell}{\ell} \right)^\ell\leq \mathrm{e}^{ \log (Lq^\ell)\cdot \ell  }. \end{equation} We now break into two cases making different choices for $\ell$ depending on the value of $q$. These cases roughly correspond to the two bounds in the lemma. First, note that by \eqref{logbdd} we  have\begin{equation}\label{eq:qlog}
	1-q \leq \log(1/q) \leq (1-q)/q. 
	\end{equation}

	\noindent\textbf{Case (i)} $\left[ 	 1/n^2\leq q \leq  1 - \frac{(\log\log n)^2 }{\log n}\right]$\textbf{:} Set $\ell =10\cdot \sqrt{\log(n) /\log(1/q)  }$, and so by \eqref{eq:qlog} we have \begin{equation}\label{eq:smallell}\ell =10\cdot \sqrt{\frac{\log n}{\log(1/q)} } \leq 10\cdot\sqrt{\frac{\log n}{1- q} } \leq  10\cdot\sqrt{\frac{\log^2 n}{(\log\log n)^2 } } =\frac{10\log n}{\log\log n}. \end{equation}
	
	The lower bound on $q$ in this case ensures that $\ell\geq 10\cdot \sqrt{\frac{\log n}{\log(n^2)} } > \mathrm{e}$, and \eqref{eq:qlog} gives \[L\cdot q^\ell = L\cdot \mathrm{e}^{-10\sqrt{\log(1/q)\log n}}\leq L\cdot \mathrm{e}^{-10\sqrt{(1-q)\log n}} \leq  L\cdot \mathrm{e}^{-10\sqrt{\frac{(\log\log n)^2}{\log n }\cdot  \log n}} \leq  (\log n)^2\cdot\frac{1}{(\log n)^{10}} ,\]for large $n$. Thus, for large $n$, $Lq^\ell< \ell$ and so \eqref{eq:tailbdd} and the bound on $\ell$ from \eqref{eq:smallell} give  \begin{equation}\Pr{ \sum_{j=1}^L X_j>\ell }   \leq \left(\frac{\mathrm{e}Lq^\ell}{\ell} \right)^\ell \leq L^\ell \cdot q^{\ell^2} \leq \left(\log^2n\right)^{\frac{10\log n}{\log\log n}} \cdot \mathrm{e}^{-100\log n} \leq n^{ -80}.   \tag*{$\lozenge$}  \end{equation}

	\noindent\textbf{Case (ii)} $\left[q \geq  1 - \frac{(\log\log n)^2 }{\log n}\right]$\textbf{:} Set $\ell  = 5\log (1-q)/\log q= 5\log\left( \frac{1}{1-q}\right)/ \log( 1/q) $. Now, \[ L\cdot q^{\ell} = \frac{100}{1-q}\left(\frac{1}{1-q} + \log n \right)\cdot \left(1-q\right)^{5}\leq 100\left(\frac{(\log\log n)^2}{\log n}\right)^3 + 100(\log n)  \left(\frac{(\log\log n)^2}{\log n}\right)^4<\frac{1}{\log n},   \] and thus $ \log(Lq^\ell)<-\log \log n$. Also by \eqref{eq:qlog} we have \begin{equation*}\ell =\frac{5\log \frac{1}{1-q}}{ \log( 1/q)}\geq  \frac{q \cdot 5\log \frac{1}{1-q}}{ 1- q} \geq  \frac{5q}{ 1- q}\cdot \log \frac{\log n}{(\log\log n)^2 }.\end{equation*}Hence, for suitably large $n$, we have \[\ell \geq 5\left(1-\frac{(\log\log n)^2 }{\log n}\right) \frac{\log n}{(\log\log n)^2 } \cdot \left(\log \log  n - 2\log\log \log n \right) \geq \frac{4\log n }{\log \log n }>\mathrm{e} . \]To conclude this case, by \eqref{eq:tailbdd} and the above, for suitably large $n$ we have 
	\begin{equation}\Pr{ \sum_{j=1}^L X_j>\ell }  \leq \mathrm{e}^{ \log (Lq^\ell)\cdot \ell  }\leq  \mathrm{e}^{- \log(\log n) \cdot \frac{4\log n }{\log \log n } }  = n^{-4}.\tag*{$\lozenge$}  \end{equation}   
It follows from Cases (i) and (ii) that for large $n$ (and the appropriate $\ell:=\ell(q)$) we have \[\Pr{\cup_{i\in [n]}\mathcal{B}(i,\ell,L)^c}\leq n \cdot \Pr{ \sum_{i=1}^L X_i>\ell }\leq n\cdot n^{-4}=n^{-3},\] by the union bound. Thus, again for suitably large $n$, combining this with \eqref{eq:probES} gives \[\Pr{\mathcal{W}(\ell,L)^c} \leq n^{-7} + \lo{n^{-9}} + n^{-3}  = \BO{n^{-3}}.\] Thus by the claim and Cases (i) and (ii) we have $\Pr{\cw(\cp(n,q))> 4(\ell+1)}=\BO{n^{-3}}$ if $q\geq n^{-2}$. All that remains is to relate the cases to the statement of the lemma. 
	
	For the first statement of this lemma, the range $0\leq q\leq 1 -(\log\log n)^2/\log n$ is covered by Case (i) and the fact that $\Pr{\cw(\cp(n,q))\geq 2}= \BO{n^{-3}}$ if $q\leq 1/n^2$ from \eqref{eq:smallq}. We now let $q=1- c(\log\log n)^2 /\log n$, where $0<c\leq 1$ is fixed. Then, for suitably large $n$, we have $\log \frac{1}{1-q}= \log\left(\frac{\log n}{c(\log\log n)^2}\right) \leq  \log\log n$. Thus, by \eqref{eq:qlog}, we have  \[  \sqrt{\frac{\log n}{\log(1/q)} }  = \frac{\sqrt{\log(1/q)\log n }}{\log(1/q)}\geq \frac{\sqrt{(1-q)\log n}}{\log(1/q)}  = \frac{\sqrt{c(\log\log n)^2}}{\log(1/q)}  \geq \sqrt{c} \cdot \frac{\log\frac{1}{1-q} }{\log(1/q)}, \]
	and so if we set $C:=C(c)=5/\sqrt{c}$ then the first statement follows from Case (ii). The second statement follows directly from Case (ii) as $ 5\log(1-q)/\log q \leq 5\log\left(\frac{1}{1-q}\right)/(1-q)$ by \eqref{eq:qlog}.\end{proof}

 \section{Conclusion \& Open Problems}\label{sec:probs}
 
 In this paper we introduced the $\cp(n,q)$ model, found a sharp threshold for the existence of cut vertices, determined the treewidth up to a $\log$ factor, and proved bounds on the vertex expansion and diameter for restricted values of $q$. Roughly speaking, the tangled path has three regimes: \textit{path-like} when $0\leq q\leq 1-\Theta(1/\log n)$, \textit{intermediate} where $ 1- \Theta(1/\log n)\leq q\leq 1- \Theta(1/n) $,  and \textit{expander} if $1-\Theta(1/n) \leq q \leq 1$. In the path-like regime there are no long edges and lots of independence, and in the expander regime one should be able to couple to the uniform case. The intermediate regime seems to be the trickiest and most interesting to analyze. 
 
 There are a wealth of open problems for $\cp(n,q)$ as one could study the effect of $q$ on almost any graph property/index of interest for sparse graphs. One fundamental problem is to determine the number of edges in $\cp(n,q)$, recalling that we disregard multi-edges. This deceptively non-trivial problem is related to clustering of consecutive numbers in Mallows permutations \cite{pinsky}. It would also be nice to close the gap for treewidth by obtaining tight bounds for all $q$.

 Theorem \ref{prop:sep} proves that $q=1- \pi^2/(6\log n)$ is a sharp threshold for containing a single vertex whose removal separates the graph into two macroscopic components. A key open problem is to determine if there is a notion of monotone property in the setting of tangled paths which guarantees the existence of a threshold (or even a sharp threshold). One candidate feature (for a property to be monotone with respect to) is the number of inversions in the permutation generating $\cp(n,q)$. However, one issue with parameterizing by the number of inversions is the fact that the tangled paths generated by $\sigma=(\sigma_1, \dots, \sigma_n)$ and its reverse $\sigma^{R}=(\sigma_n, \dots, \sigma_1)$ are isomorphic, but the number of inversions may differ greatly as $\inv(\sigma^R)=\binom{n}{2}-\inv(\sigma)$.
 \section*{Acknowledgements}
 The authors thank the anonymous referees for reading the paper carefully. In particular one referee found a mistake in the proof of Lemma \ref{lem:largeqpatterncouple}, and correcting this led the authors to a strengthening of Theorems \ref{treebanddiam} and \ref{thm:treepathcut}. All authors of this work were supported by EPSRC project EP/T004878/1: Multilayer Algorithmics to Leverage Graph Structure.  No data were created or analysed in this study.  An extended abstract appeared here \cite{euro}.

\bibliography{ref}
\bibliographystyle{plainurl}

\end{document}